\documentclass[10pt]{amsart}
\usepackage{latexsym}
\usepackage{amsmath}
\usepackage{amssymb}
\usepackage{mathrsfs}
\usepackage{graphicx}
\usepackage{color}
\usepackage{pgfpages}
\usepackage{ifthen}
\usepackage{leftidx,tensor}
\usepackage[T1]{fontenc}
\usepackage[latin1]{inputenc}
\usepackage{mathtools}
\usepackage{comment}
\usepackage{dsfont}
\usepackage[nocompress]{cite}

\usepackage[shortlabels]{enumitem}
\usepackage{aliascnt}
\usepackage[bookmarks=true,pdfstartview=FitH, pdfborder={0 0 0}, colorlinks=true,citecolor=red, linkcolor=blue]{hyperref}
\usepackage{bbm}
\usepackage{nicefrac}

\theoremstyle{plain}
\newtheorem{thm}{Theorem}[section]
\newaliascnt{cor}{thm}
\newaliascnt{prop}{thm}
\newaliascnt{lem}{thm}
\newtheorem{cor}[cor]{Corollary}
\newtheorem{prop}[prop]{Proposition}
\newtheorem{lem}[lem]{Lemma}
\aliascntresetthe{cor}
\aliascntresetthe{prop}
\aliascntresetthe{lem} 
\theoremstyle{definition}
\newaliascnt{defn}{thm}
\newaliascnt{asu}{thm}
\newaliascnt{con}{thm}
\aliascntresetthe{defn}
\aliascntresetthe{asu}
\aliascntresetthe{con}
\newcounter{stp}
\newcounter{stpi}
\newcounter{stpci}
\newcounter{stpiii}

\theoremstyle{remark}
\newaliascnt{rem}{thm}
\newaliascnt{exa}{thm}
\newaliascnt{masu}{thm}
\newaliascnt{nota}{thm}
\newaliascnt{sett}{thm}
\newtheorem{rem}[rem]{Remark}

\aliascntresetthe{rem}
\aliascntresetthe{exa}
\aliascntresetthe{masu}
\aliascntresetthe{nota}
\aliascntresetthe{sett}
\numberwithin{equation}{section}

\setlist[enumerate]{font = \normalfont}

\newcommand{\F}{\mathbb{F}}
\newcommand{\E}{\mathbb{E}}
\newcommand{\R}{\mathbb{R}}
\newcommand{\C}{\mathbb{C}}

\newcommand{\N}{\mathbb{N}}
\newcommand{\bP}{\mathbb{P}}
\newcommand{\T}{\mathbb{T}}
\newcommand{\D}{\mathbb{D}}

\newcommand{\bS}{\mathbb{S}}

\newcommand{\tE}{\Tilde{\mathbb{E}}}

\newcommand{\rW}{\mathrm{W}} 
\newcommand{\rL}{\mathrm{L}}
\newcommand{\rE}{\mathrm{E}}
\newcommand{\rH}{\mathrm{H}}
\newcommand{\rT}{\mathrm{T}}
\newcommand{\rC}{\mathrm{C}}
\newcommand{\rB}{\mathrm{B}}

\newcommand{\rN}{\mathrm{N}}
\newcommand{\rX}{\mathrm{X}}
\newcommand{\rY}{\mathrm{Y}}
\newcommand{\rJ}{\mathrm{J}}
\newcommand{\rI}{\mathrm{I}}
\newcommand{\rII}{\mathrm{II}}
\newcommand{\rIII}{\mathrm{III}}
\newcommand{\rIV}{\mathrm{IV}}
\newcommand{\rV}{\mathrm{V}}
\newcommand{\rVI}{\mathrm{VI}}
\newcommand{\rVII}{\mathrm{VII}}
\newcommand{\rVIII}{\mathrm{VIII}}
\newcommand{\rIX}{\mathrm{IX}}
\newcommand{\BUC}{\mathrm{BUC}}
\newcommand{\SO}{\mathrm{SO}}
\newcommand{\rd}{\,\mathrm{d}}

\newcommand{\mre}{\mathrm{e}}
\newcommand{\avg}{\mathrm{avg}}

\newcommand{\dist}{\mathrm{dist}}
\newcommand{\Cmr}{C_\mathrm{MR}}

\newcommand{\Cexp}{C_\mathrm{exp}}

\newcommand{\mS}{m_\mathcal{S}}
\newcommand{\cS}{\mathcal{S}}
\newcommand{\cF}{\mathcal{F}}
\newcommand{\cE}{\mathcal{E}}
\newcommand{\cO}{\mathcal{O}}
\newcommand{\cB}{\mathcal{B}}
\newcommand{\cL}{\mathcal{L}}
\newcommand{\cM}{\mathcal{M}}
\newcommand{\cN}{\mathcal{N}}
\newcommand{\cG}{\mathcal{G}}
\newcommand{\cQ}{\mathcal{Q}}
\newcommand{\cK}{\mathcal{K}}
\newcommand{\cJ}{\mathcal{J}}
\newcommand{\cT}{\mathcal{T}}
\newcommand{\cR}{\mathcal{R}}
\newcommand{\tcK}{\Tilde{\cK}}
\newcommand{\tv}{\Tilde{v}}
\newcommand{\td}{\Tilde{d}}
\newcommand{\tp}{\Tilde{p}}
\newcommand{\tell}{\Tilde{\ell}}
\newcommand{\tomega}{\Tilde{\omega}}
\newcommand{\tz}{\Tilde{z}}
\newcommand{\hv}{\widehat{v}}
\newcommand{\hd}{\widehat{d}}
\newcommand{\hp}{\widehat{p}}
\newcommand{\hell}{\widehat{\ell}}
\newcommand{\homega}{\widehat{\omega}}
\newcommand{\hz}{\widehat{z}}

\newcommand{\ovd}{\overline{d}}
\newcommand{\eps}{\varepsilon}
\newcommand{\del}{\partial}

\DeclareMathOperator{\Id}{Id}
\DeclareMathOperator{\mdiv}{div}
\newcommand{\tin}{\enspace \text{in} \enspace}
\newcommand{\ton}{\enspace \text{on} \enspace}
\newcommand{\tfor}{\enspace \text{for} \enspace}
\newcommand{\tforall}{\enspace \text{for all} \enspace}
\newcommand{\tand}{\enspace \text{and} \enspace}
\newcommand{\tso}{\enspace \text{so} \enspace}
\newcommand{\tif}{\enspace \text{if} \enspace}

\newcommand{\twith}{\enspace \text{with} \enspace}
\newcommand{\taswellas}{\enspace \text{as well as} \enspace}

\begin{document}

\title[Interaction of liquid crystals with a rigid body]{Interaction of liquid crystals with a rigid body}


\author{Tim Binz}
\address{Princeton University,
Program in Applied \& Computational Mathematics, Fine Hall, Washington Road, 08544 Princeton, NJ, USA.}
\email{tb7523@princeton.edu}
\author{Felix Brandt}
\address{Technische Universit\"{a}t Darmstadt,
Schlo\ss{}gartenstra{\ss}e 7, 64289 Darmstadt, Germany.}
\email{brandt@mathematik.tu-darmstadt.de}
\email{hieber@mathematik.tu-darmstadt.de}
 \email{royarnab244@gmail.com}
\author{Matthias Hieber}
\author{Arnab Roy}

\begin{abstract}
This article investigates the interaction of nematic liquid crystals modeled by a simplified Ericksen-Leslie model with a rigid body.
It is shown that this problem is locally strongly well-posed, and that it also admits a unique, global strong solution for initial data close to constant equilibria.
The proof of the global strong solution relies on  a new splitting method for the director in a mean value zero and average part.
\end{abstract}

\maketitle

\section{Introduction}

Ericksen \cite{Ericksen:62} and Leslie \cite{Leslie:68} pioneered the development of a continuum theory for the flow of nematic liquid crystals. Their theory models the flow of nematic crystals 
from a fluid dynamical point of view and describes the evolution of the underlying system by a certain fluid equation coupled to the configuration $d$ of rod-like crystals. The original 
derivation is based on the 
conservation laws for mass and linear momentum as well as on the constitutive relations given by Ericksen \cite{Ericksen:62} and Leslie \cite{Leslie:68}. The system subject to general Leslie and 
general Ericksen stress is quite involved, and many analytical questions related to the resulting model remain open until today.

A simplified system was introduced and analyzed by Lin \cite{Lin:89} as well as by Lin and Liu \cite{LL:95}.
They studied first the situation where the nonlinearity in the equation for the director $d$ is replaced by a 
Ginzburg-Landau term and proved the existence of global weak solutions under suitable assumptions. Considering the original situation where the nonlinear terms describing  the evolution of the 
director $d$  are kept in their original form and are not replaced by Ginzburg-Landau terms, a rather complete understanding of the 
well-posedness and the dynamics of the simplified system in the strong sense was obtained in \cite{HNPS:16}. In particular, it was shown that the condition $|d|=1$ is preserved by the system for all 
$t>0$. 

Only recently, a thorough analysis of the Ericksen-Leslie system subject to the general Ericksen and Leslie stresses was obtained by Hieber and Pr\"uss \cite{HP:16, HP:17, HP:19}. For related results 
concerning various assumptions on the Leslie as well as on the Ericksen coefficients, we refer to  Lin \cite{LL:01}, Lin and Liu \cite{LL:16}, and  Wu, Xu and Liu\cite{WXL:13} as well as to the 
survey articles  \cite{HP:18} and \cite{WZZ:21}.

The present article investigates for the first time the interaction problem of a rigid body with nematic liquid crystals described by a simplified Ericksen-Leslie model as considered in \cite{HNPS:16}. It 
establishes the existence and uniqueness of a strong solution, both, locally-in-time and globally-in-time for initial data close to equilibria. One essential feature of our approach is that we do not 
use the Ginzburg-Landau energy functional, but we show that the director condition is preserved in the interaction problem.
Moreover, the equilibria are deduced from an investigation of the energy functional.

The main difficulties in our stability analysis of the equilibria arise from the lack of autonomy of the transformed system on a fixed domain as well as the lack of invertibility of the linearization, so classical tools 
such as principles of linearized stability fall short in this case. Instead, inspired by the work of Haak, Maity, Takahashi and Tucsnak \cite{HMTT:19} in the context of  compressible fluids, we split 
the equations satisfied by the director 
into one part with mean value zero and the associated average part. This enables us to handle the linearized problem on the whole time interval $\R_+$ and paves the way to the 
global well-posedness result with certain decay properties.

The  interaction of a  rigid body immersed in an incompressible Newtonian fluid without a liquid crystal component is, of course, a very classical topic.  
We refer here  only to the works of  Desjardins and Esteban \cite{DE:99, DE:00}, Conca, San Mart\'in, Tucsnak \cite{CSMT:00},  Ne{\v{c}}asov{\'a} et al.\ \cite{nevcasova2021}, 
Feireisl, Roy and Zarnescu \cite{FRZ:23} in the weak setting and to  Galdi \cite{Galdi:02}, Galdi and Silvestre \cite{GS:02}, Takahashi \cite{Takahashi:03}, Cumsille and Takahashi \cite{CT:08},  
Geissert, G\"otze and Hieber \cite{GGH:13}, Kaltenbacher, Kukavica, Lasiecka, Triggiani, Tuffaha and Webster \cite{KKLTTW:18}, Maity and Tucsnak \cite{MT:18}  and Ervedoza, Maity and Tucsnak 
\cite{EMT:23} in the strong setting. 

For related results in the context of elastic structures, we refer for example to the articles of  Coutand and Shkoller \cite{CS:05},  Grandmont, Hillairet and Lequeurre 
\cite{GHL:19}; for results  in the  non-viscous case let us refer for example to  the work of Glass, Lacave and Sueur \cite{GLS:16}. We mention here also the recent article 
\cite{geng2023global} on weak solutions to the interaction of a rigid body with liquid crystals described by the $Q$-tensor model, which is fundamentally different from the Ericksen-Leslie model.  
      
Let us emphasize that the interaction problem under consideration is different from the ones associated to incompressible Newtonian fluids due to the elasticity component in the stress tensor and its  
interaction with the rigid body. 

The physical motivation to study liquid crystal-rigid body interaction results from so-called liquid crystal colloids formed by dispersion of colloidal particles in the liquid crystal host medium.
The rigid body in our analysis can be viewed as a colloidal particle. For further details on the physical background, we refer e.g.\ to \cite{SEN,MU} or \cite{MO}.

The first step for solving the interaction problem is the transformation of the moving domain problem to a fixed domain problem by means of a local change of coordinates due to Inoue and Wakimoto \cite{IW:77}.
In return, additional time-dependent terms arise, and the transformed system is of a more complicated shape.
The proof of the local strong well-posedness relies on the maximal regularity of the linearized interaction problem.

With regard to the global result, we require maximal $\rL^p([0,\infty))$-regularity of the linearization around constant equilibria.
In fact, we obtain a maximal $\rL^p([0,\infty))$-regularity type estimate with exponential decay of the complete system of equations by invoking the fluid-structure semigroup from \cite{MT:18} and performing the splitting of the equations satisfied by the director into a part with mean value zero in the spatial domain and the remaining average part.

The proofs are completed by Lipschitz estimates of the nonlinear terms for which we rely on appropriate estimates of the terms associated to the coordinate transform in the local and global setting.
In particular, we verify that no collision of the rigid body and the outer boundary of the viscous incompressible fluid with active liquid crystals can occur provided the rigid body starts with an initial distance from the boundary, and the initial values are taken sufficiently close to the equilibria.

This paper is organized as follows.
In \autoref{sec:main results}, the model for the interaction problem is introduced, and the main results on the local strong well-posedness as well as the global strong well-posedness close to constant equilibria are stated.
The aim of \autoref{sec:dir cond energy equilibria} is twofold:
On the one hand, it is shown that the director condition is preserved in the interaction problem.
On the other hand, this section discusses the energy of the interaction problem and deduces the equilibrium solutions from there.
\autoref{sec:change of var} is dedicated to the transformation of the interaction problem to a fixed domain, and the main results are reformulated in this reference configuration.
In \autoref{sec:lin theory}, we establish the maximal regularity of the linearized problem.
Moreover, we show the maximal $\rL^p([0,\infty))$-regularity of the linearization around constant equilibria up to a shift.
\autoref{sec:proof local well-posedness} presents the proof of the local strong well-posedness by a fixed point argument.
For this, estimates of the nonlinear terms based on a good control of the terms resulting from the coordinate transform are proved.
In \autoref{sec:proof global strong wp}, we first introduce the maximal regularity type result in order to get exponential decay and then conclude the proof of the global strong well-posedness close to constant equilibria from the contraction mapping principle by invoking adjusted estimates of the transform.

\section{Description of the model and main results}\label{sec:main results}

Let $\cO \subset \R^3$ denote a bounded domain with boundary of class $\rC^3$, and let $0 < T \le \infty$ represent a positive time.
We consider a rigid body moving inside $\cO$ which is assumed to be closed, bounded and simply connected.
Furthermore, the domain occupied by the rigid body is denoted by $\cS(t)$ for $t \in (0,T)$, and the remaining part of the domain $\cF(t) \coloneqq \cO \setminus \cS(t)$ is assumed to be filled by a viscous incompressible fluid with active liquid crystals.
The initial domain of the rigid body is denoted by $\cS_0$.
Besides, we suppose that its boundary $\del \cS_0$ is of class $\rC^3$ as well.
Then $\cF_0 =\cO \setminus \cS_0$ is the initial fluid domain.
The solid domain $\cS(t)$ at time $t$ is given by 
\begin{equation}\label{eq:solid dom}
    \cS(t)= \left\{h(t)+ Q(t)y : y\in \cS_{0}\right\}, 
\end{equation}
where $h(t)$ represents the center of mass of the body, and $Q(t) \in \SO(3)$ is associated to the rotation of the rigid body.
Moreover, since $Q'(t)Q^{\top}(t)$ is skew symmetric, for each $t>0$, we can represent this matrix by a unique vector $\Omega(t)\in \R^3$ such that
\begin{equation*}
    Q'(t)Q^{\top}(t)y=\Omega (t)\times y, \tforall y\in \R^3.
\end{equation*}
By $\cF_T$, $\cS_T$ and $\del \cS_T$, we denote the spacetime related to $\cF(t)$, $\cS(t)$ and $\del \cS(t)$, respectively, i.e.,
\begin{equation*}
    \cF_T \coloneqq \{(t,x) : t \in (0,T), \, x \in \cF(t)\}, \enspace \cS_T \coloneqq \{(t,x) : t \in (0,T), \,x \in \cS(t)\}, \tand
\end{equation*}
\begin{equation*}
    \del \cS_T 
        \coloneqq \{(t,x) : t \in (0,T), \, x \in \del \cS(t)\}.
\end{equation*}
The velocity of the rigid body is described by
\begin{equation*} 
    u^{\cS}(t,x)= h'(t)+  \Omega(t) \times (x-h(t)), \tforall (t,x) \in \cS_T, 
\end{equation*}
where $h' \colon (0,T) \to \R^3$ is the translational velocity of the center of mass, and $\Omega \colon (0,T) \to \R^3$ denotes the angular velocity of the body.  

In the context of the Ericksen-Leslie model for nematic liquid crystals, the model variables are the fluid velocity $u^{\cF} \colon \cF_T \to \R^3$, the fluid pressure $\pi^{\cF} \colon \cF_T \to \R$, and the director $\ovd \colon \cF_T \to \bS^2$, where $\bS^2=\left\{x\in\R^3 : |x|=1\right\}$.
We set $\D = \nicefrac{1}{2}(\nabla u^{\cF} + (\nabla u^{\cF})^\top)$. The stress tensor in the situation of the simplified isothermal Ericksen-Leslie model under consideration is given by $\T=\T_N+ \T_E$, where
\begin{equation*}
    \T_N = 2\mu \D, \tand \T_E 
        = -\lambda\nabla \ovd \left(\nabla \ovd\right)^{\top}, \tso \T= 2\mu  \D -\lambda\nabla d (\nabla d)^{\top},
\end{equation*}
and $\mu > 0$ and $\lambda > 0$ are constants.
Upon assuming that the constant fluid density $\rho^{\cF}$ equals one, and for another constant $\gamma > 0$, we obtain the following \textit{simplified isothermal Ericksen-Leslie model} 
\begin{equation}\label{eq:LC-fluid equationsiso}
\left\{
    \begin{aligned}
        {\partial_t u^{\cF}} + \left(u^{\cF}\cdot \nabla\right)u^{\cF} + \nabla \pi^{\cF} &=\mu\Delta u^{\cF}  - \lambda \mdiv\left(\nabla \ovd \left(\nabla \ovd\right)^{\top}\right), &&\tin \cF_T ,\\
        \mdiv u^{\cF}&=0, &&\tin \cF_T ,\\
        \gamma\left[\partial_t \ovd + \left(u^{\cF}\cdot \nabla\right)\ovd\right] -\lambda \Delta \ovd &= \lambda |\nabla \ovd|^2 \ovd, &&\tin \cF_T .
    \end{aligned}
\right.
\end{equation}

With regard to the rigid body, let $\rho_{\cS}$ denote its density.
For simplicity, we assume that $\rho_{\cS} \equiv 1$ in the sequel, yielding that $\mS = \int_{\cS(t)} \rd x$ is its total mass.
In addition, without loss of generality, we suppose that the center of mass of the rigid body starts in the origin, i.e., $h(0) = 0$. Let $J(t)$ be the inertia matrix and the initial value of the inertial matrix is $J_0$.  Here $J_0$ and $J(t)$ are given by
\begin{equation*}
    J_0 =  \int\limits_{\cS_0}\left(|x|^2\Id - x\otimes x\right) \rd x \tand J(t) =  \int\limits_{\cS(t)}
 \left(|x-h(t)|^2\Id - (x-h(t))\otimes (x-h(t))\right)\rd x,
\end{equation*}
respectively.
Consequently, $J(t)$ is symmetric positive definite and satisfies
\begin{equation}\label{eq:Syl and Jab}
    J(t)=Q(t)J_0 Q^{\top}(t) \taswellas J(t)a\cdot b=\int\limits_{\cS_0} (a \times Q(t)y)\cdot (b\times Q(t)y)\rd y, \tforall \ a,b \in \R^3.
\end{equation}
For the underlying total stress tensor
\begin{equation}\label{eq:stress tensor}
    \Sigma(u^{\cF},\pi^{\cF}, \ovd)=\left(\T-\pi^{\cF}\Id\right)=2\mu  \D -\lambda\nabla \ovd \left(\nabla \ovd\right)^{\top}-\pi^{\cF}\Id,
\end{equation}
the motion of the rigid body can be described by Newton's law, i.e., 
\begin{equation}\label{eq:rigid body eqs intro}
    \mS h''(t)
    = -\int\limits_{\partial \cS(t)} \Sigma(u^{\cF},\pi^{\cF}, \ovd) \nu(t)\rd \Gamma , \tand 
    (J\Omega)' (t) = -\int\limits_{\partial \cS(t)} (x-h(t)) \times  \Sigma(u^{\cF},\pi^{\cF}, \ovd) \nu(t)\rd \Gamma,
\end{equation}
where $\nu(t)$ denotes the unit outward normal to the boundary of $\cF(t)$, i.e., it is directed towards $\cS(t)$.

We assume that the fluid velocity satisfies homogeneous boundary conditions on the outer boundary and coincides with the velocity of the rigid body on their common interface, while the director is supposed to fulfill homogeneous Neumann boundary conditions, i.e.,
\begin{equation}\label{eq:bdry cond}
    u^{\cF} =0, \enspace \partial_\nu \ovd = 0, \ton (0,T) \times \del \cO, \tand u^{\cF} = u^{\cS}, \enspace \partial_\nu \ovd = 0, \ton \del \cS_T.
\end{equation}
The system is completed by the initial conditions
\begin{equation}\label{eq:initial}
    u^{\cF}(0,x)=v_0(x),\enspace \ovd(0,x)=d_0(x) \tin \cF_0,\enspace h(0)=0, \enspace h'(0)=\ell_0, \tand \Omega(0)=\omega_0.
\end{equation}

For the statements of the main theorems, we require the precise notion of solutions, so we first define function spaces on time-dependent domains.
In fact, let $X \colon \cF_T \to \cF_0$ be a map such that
\begin{equation*}
    \varphi \colon \cF_T \to (0,T) \times \cF_0, \enspace (t,x) \mapsto (t,X(t,x))
\end{equation*}
is a $\rC^1$-diffeomorphism, and $X(\tau,\cdot) \colon \cF(\tau) \to \cF_0$ are $\rC^2$-diffeomorphisms for all $\tau \in [0,T]$.
We then define
\begin{equation*}
    \rW^{s,p}(0,T;\rW^{l,q}(\cF(\cdot))) \coloneqq \{f(t,\cdot) \colon \cF(t) \to \R : f \circ \varphi \in \rW^{s,p}(0,T;\rW^{l,q}(\cF_0))\}
\end{equation*}
for $p,q \in (1,\infty)$, $s \in \{0,1\}$ and $l \in \{0,1,2\}$.
By $\rL_0^q(\cF_0)$, we will denote the functions in $\rL^q(\cF_0)$ with mean value zero.
Next, we settle some notation with regard to the initial data.
For this purpose, let us recall the space of solenoidal vector fields on $\rL^q(\cF_0)^3$ defined by $\rL_\sigma^q(\cF_0) := \bP \rL^q(\cF_0)^3$, where $\bP$ represents the Helmholtz projection on $\cF_0$.
Denoting the principle variable by $z_0 = (v_0,d_0,\ell_0,\omega_0)$, we then define
\begin{equation}\label{eq:trace space}
    \begin{aligned}
        \rX_\gamma \coloneqq \bigl\{
        &z_0 \in \rB_{qp}^{2-\nicefrac{2}{p}}(\cF_0)^3 \cap \rL_\sigma^q(\cF_0) \times \rB_{qp}^{2-\nicefrac{2}{p}}(\cF_0)^3 \times \R^3 \times \R^3 : v_0 = 0, \enspace \del_\nu d_0 = 0, \ton \del \cO, \tand\\
        &v_0 = \ell_0 + \omega_0 \times y, \enspace \del_\nu d_0 = 0, \ton \del \cS_0 \bigr\},
    \end{aligned}
\end{equation}
where $\rB_{qp}^{s}(\cF_0)$ denotes the Besov space with parameters $p, q \in (1,\infty)$ and $s \in \R$.
The fact that the trace and the normal derivative are defined is guaranteed by the following conditions.
Let $p,q \in (1,\infty)$ fulfill
\begin{equation}\label{eq:cond p and q}
    \frac{2}{p} + \frac{3}{q} < 1.
\end{equation}
We observe that \eqref{eq:cond p and q} yields the embedding $\rB_{qp}^{2-\nicefrac{2}{p}}(\cF_0) \hookrightarrow \rC^1(\overline{\cF_0})$, see \cite[Theorem~4.6.1]{Tri:78}.
The first main result of this paper on the local strong well-posedness of the interaction problem reads as follows.

\begin{thm}\label{thm:local wp}
Let $p,q \in (1,\infty)$ be such that \eqref{eq:cond p and q} holds true, let $\cO \subset \R^3$ be a bounded domain of class $\rC^3$, and consider the domains $\cS_0$ and $\cF_0$ occupied by the rigid body and the fluid at time zero.
Furthermore, let $z_0 = (v_0,d_0,\ell_0,\omega_0) \in \rX_\gamma$, where $\rX_\gamma$ is defined in \eqref{eq:trace space} and $|d_0| = 1$.
If for some $r > 0$, it is valid that $\dist(\cS_0,\partial \cO) > r$, then there are $T' \in (0,T]$ and $X \in \rC^1\left([0,T'];\rC^2(\R^3)^3\right)$ such that $X(\tau,\cdot) \colon \cF(\tau) \to \cF_0$ are $\rC^2$-diffeomorphisms for all $\tau \in [0,T']$, and the interaction problem \eqref{eq:LC-fluid equationsiso}--\eqref{eq:initial} admits a unique solution
\begin{equation*}
    \begin{aligned}
        u^{\cF} &\in \rW^{1,p}\left(0,T';\rL^q(\cF(\cdot))^3\right) \cap \rL^p\left(0,T';\rW^{2,q}(\cF(\cdot))^3\right), \enspace \pi^{\cF} \in \rL^p\left(0,T';\rW^{1,q}(\cF(\cdot)) \cap \rL_0^q(\cF(\cdot))\right),\\
        \ovd &\in \rW^{1,p}\left(0,T';\rL^q(\cF(\cdot))^3\right) \cap \rL^p\left(0,T';\rW^{2,q}(\cF(\cdot))^3\right),\\
        h' &\in \rW^{1,p}(0,T')^3, \enspace h \in \rL^\infty(0,T')^3, \tand \Omega \in \rW^{1,p}(0,T')^3.
    \end{aligned}
\end{equation*}
\end{thm}

\begin{rem}\label{rem:notion strong solution and forcing terms}
\begin{enumerate}[(a)]
    \item We refer to a solution in the regularity class as in \autoref{thm:local wp} as a \textit{strong solution}.
    \item We can reproduce a similar local-in-time result for forcing terms.
    For instance, one can consider external forces and torques in the equations accounting for the motion of the rigid body \eqref{eq:rigid body eqs intro}.
    An analogous result as \autoref{thm:local wp} is then obtained if the forcing terms are in $\rL^p(0,T)^3$.
    \item Let us observe that the uniqueness of the pressure is a consequence of the choice of the intersection with $\rL_0^q(\cF(\cdot))$ in the spatial component.
\end{enumerate}    
\end{rem}

It is easy to verify that $z_* = (0,d_*,0,0)$, with $|d_*| = 1$ and $d_*$ constant is a stationary solution of the interaction problem \eqref{eq:LC-fluid equationsiso}--\eqref{eq:initial}.
The second main result in this paper asserts the global strong well-posedness of the interaction problem for initial data close to such constant equilibria.

\begin{thm}\label{thm:global wp close to equilibria}
Let $p,q \in (1,\infty)$ be such that \eqref{eq:cond p and q} holds true, let $\cO \subset \R^3$ be a bounded domain of class $\rC^3$, and consider the domains $\cS_0$ and $\cF_0$ occupied by the rigid body and the fluid at time zero.
Moreover, assume that $\dist(\cS_0,\partial \cO) > r$ for some $r > 0$, and let $d_* \in \R^3$ constant with $|d_*| = 1$.
Then there is $\eta_0 > 0$ such that for all $\eta \in (0,\eta_0)$, there are two constants $\delta_0 > 0$ and $C > 0$ such that for all $\delta \in (0,\delta_0)$ and for all $z_0 = (v_0,d_0,\ell_0,\omega_0) \in \rX_\gamma$, with $\rX_\gamma$ as in \eqref{eq:trace space}, $|d_0| = 1$ and $\| (v_0, d_0 - d_*, \ell_0,\omega_0) \|_{\rX_\gamma} \le \delta$, the interaction problem \eqref{eq:LC-fluid equationsiso}--\eqref{eq:initial} admits a unique strong solution $(u^\cF,\pi^{\cF},\ovd,h',\Omega)$ such that
\begin{equation*}
    \begin{aligned}
        &\quad \left\| \mre^{\eta(\cdot)} u^\cF \right\|_{\rL^p(0,\infty;\rW^{2,q}(\cF(\cdot)))} + \left\| \mre^{\eta(\cdot)} \partial_t u^\cF \right\|_{\rL^p(0,\infty;\rL^q(\cF(\cdot)))} + \left\| \mre^{\eta(\cdot)} u^\cF \right\|_{\rL^\infty\left(0,\infty;\rB_{qp}^{2-\nicefrac{2}{p}}(\cF(\cdot))\right)}\\
        &+ \left\| \mre^{\eta(\cdot)} \pi^\cF \right\|_{\rL^p(0,\infty;\rW^{1,q}(\cF(\cdot)) \cap \rL_0^q(\cF(\cdot)))} + \left\| \mre^{\eta(\cdot)} \nabla \ovd \right\|_{\rL^p(0,\infty;\rL^q(\cF(\cdot)))} + \left\| \mre^{\eta(\cdot)} \nabla^2 \ovd \right\|_{\rL^p(0,\infty;\rL^q(\cF(\cdot)))}\\
        &+ \left\| \mre^{\eta(\cdot)} \partial_t \ovd \right\|_{\rL^p(0,\infty;\rL^q(\cF(\cdot)))} + \left\| \ovd - d_* \right\|_{\rL^\infty\left(0,\infty;\rB_{qp}^{2-\nicefrac{2}{p}}(\cF(\cdot))\right)} + \left\| \mre^{\eta(\cdot)} h' \right\|_{\rW^{1,p}(0,\infty)}
        + \left\| h \right\|_{\rL^\infty(0,\infty)}\\
        &+ \left\| \mre^{\eta(\cdot)} \Omega \right\|_{\rW^{1,p}(0,\infty)} \le C \delta.
    \end{aligned}
\end{equation*}
In addition, we have $\dist(\cS(t),\partial \cO) > \nicefrac{r}{2}$ for all $t \in [0,\infty)$.
\end{thm}

The following result discusses the convergence of the solution to the equilibria in the trace space at an exponential rate.
It is an immediate consequence of \autoref{thm:global wp close to equilibria}.

\begin{cor}\label{cor:exp conv to equilibira}
Under the assumptions of \autoref{thm:global wp close to equilibria}, for a constant $C > 0$ independent of $t > 0$, the solution $(u^\cF,\pi^{\cF},\ovd,h',\Omega)$ to the interaction problem \eqref{eq:LC-fluid equationsiso}--\eqref{eq:initial} especially satisfies
\begin{equation*}
    \left\| u^{\cF}(t,\cdot) \right\|_{\rB_{qp}^{2-\nicefrac{2}{p}}(\cF(t))} + \| h'(t) \|_{\R^3} + \| \Omega(t) \|_{\R^3} \le C \delta \mre^{-\eta t}, \tforall t > 0.
\end{equation*}
\end{cor}

\section{The director condition, energy and equilibria}\label{sec:dir cond energy equilibria}

We verify that the director part $\ovd$ of a strong solution to the interaction problem \eqref{eq:LC-fluid equationsiso}--\eqref{eq:initial} always satisfies the condition $|\ovd(t,x)|=1$ in \autoref{ssec:dir cond}.
\autoref{ssec:energy} is dedicated to showing that the energy of the interaction system is non-increasing, and to determining the set of equilibrium solutions. 

\subsection{Preservation of the director condition}\label{ssec:dir cond}
\

Due to the physical interpretation of the director $\ovd$, it is natural to include the condition 
\begin{equation}\label{cond:d}
    |\ovd(t,x)|=1, \tforall (t,x)\in \cF_T.
\end{equation}
Assuming $|d_0|=1$, we can prove that \eqref{cond:d} is preserved in view of the equation \eqref{eq:LC-fluid equationsiso}$_3$ satisfied by $\ovd$.

\begin{prop}\label{prop:preservation of the dir cond}
Let $p,q \in (1,\infty)$ satisfy \eqref{eq:cond p and q}, and let $z_0 =(v_0,d_0,\ell_0,\omega_0) \in \rX_\gamma$, for $\rX_\gamma$ as introduced in \eqref{eq:trace space}, be such that $|d_0| = 1$.
Moreover, let $(u^{\cF},\pi^{\cF},\ovd,h',\Omega)$ be a solution to \eqref{eq:LC-fluid equationsiso}--\eqref{eq:initial} satisfying
\begin{equation*}
    (u^{\cF},\ovd)\in \rW^{1,p}\left(0,T;\rL^q(\cF(\cdot))^3\right) \cap \rL^p\left(0,T;\rW^{2,q}(\cF(\cdot))^3\right)\times \rW^{1,p}\left(0,T;\rL^q(\cF(\cdot))^3\right) \cap \rL^p\left(0,T;\rW^{2,q}(\cF(\cdot))^3\right).
\end{equation*}
Then $\left|\ovd(t,x)\right|\equiv 1$ holds for all $(t,x)\in \cF_T$.
\end{prop}
\begin{proof}
Let us multiply  \eqref{eq:LC-fluid equationsiso}$_3$ by $2\ovd(|\ovd|^2-1)$ and integrate over $\cF(t)$ so that
\begin{equation}\label{d:mult}
    \int\limits_{\cF(t)}\left(\gamma\left[{\partial_t \ovd} + (u^{\cF}\cdot \nabla) \ovd\right] - \lambda\Delta \ovd  - \lambda |\nabla \ovd|^2 \ovd\right)\cdot 2\ovd\left(|\ovd|^2-1\right) \rd x=0.
\end{equation}
\noindent{\textit{ First term of \eqref{d:mult}.}} 
We use integration by parts and the Reynolds transport theorem to obtain
\begin{equation}\label{eq:T1}
    \begin{aligned}
        &\quad \int\limits_{\cF(t)} \partial_t \ovd\cdot 2\ovd\left(|\ovd|^2-1\right)\rd x = \int\limits_{\cF(t)} \partial_t \ovd\cdot 2\ovd |\ovd|^2\rd x - \int\limits_{\cF(t)} \partial_t \ovd\cdot 2\ovd\rd x \\
        &= \int\limits_{\cF(t)} \frac{1}{2}\partial_t \left(|\ovd|^2|\ovd|^2\right)\rd x - \int\limits_{\cF(t)}\partial_t \left(|\ovd|^2\right)\rd x\\
        &=\frac{1}{2}\frac{\mathrm{d}}{\mathrm{d}t}\int\limits_{\cF(t)}\left(|\ovd|^2|\ovd|^2\right)\rd x-\frac{1}{2}\int\limits_{\del \cF(t)}  \left(|\ovd|^2|\ovd|^2\right) \left(u^{\cF}\cdot n\right)\rd \Gamma - \frac{\mathrm{d}}{\mathrm{d}t}\int\limits_{\cF(t)} |\ovd|^2 \rd x + \int\limits_{\del \cF(t)} |\ovd|^2\left(u^{\cF}\cdot n\right)\rd \Gamma\\
        &= \frac{\mathrm{d}}{\mathrm{d}t}\int\limits_{\cF(t)} \left(\frac{1}{2}\left(|\ovd|^2|\ovd|^2\right) - |\ovd|^2 +\frac{1}{2}\right)\rd x -\frac{1}{2}\int\limits_{\del\cF(t)}  (|\ovd|^2|\ovd|^2) \left(u^{\cF}\cdot n\right)\rd \Gamma + \int\limits_{\del\cF(t)} |\ovd|^2\left(u^{\cF}\cdot n\right)\rd \Gamma.
    \end{aligned}
\end{equation}
\noindent{\textit{ Second term of \eqref{d:mult}.}} We use $\mdiv u^{\cF}=0\ \mbox{ in }  \cF(t)$ and integrate by parts again to get
\begin{equation}\label{eq:T2}
    \begin{aligned}
        \int\limits_{\cF(t)} \left(u^{\cF}\cdot \nabla\right) \ovd \cdot 2\ovd\left(|\ovd|^2-1\right)\rd x 
        &= \int\limits_{\cF(t)} 2\left(u^{\cF}\cdot \nabla\right) \ovd \cdot \ovd |\ovd|^2\rd x - 2\int\limits_{\cF(t)} \left(u^{\cF}\cdot \nabla\right) \ovd \cdot \ovd\rd x \\
        &= \int\limits_{\cF(t)} (u^{\cF}\cdot \nabla) |\ovd|^2 |\ovd|^2\rd x - \int\limits_{\cF(t)} (u^{\cF}\cdot \nabla) |\ovd|^2\rd x\\
        &= \frac{1}{2}\int\limits_{\del\cF(t)}  (|\ovd|^2|\ovd|^2) (u^{\cF}\cdot n)\rd \Gamma - \int\limits_{\del\cF(t)} |\ovd|^2(u^{\cF}\cdot n)\rd \Gamma.
    \end{aligned}
\end{equation}
\noindent{\textit{Third term of \eqref{d:mult}.}} Another integration by parts leads to
\begin{equation}\label{eq:T3}
    \begin{aligned}
        -2\lambda\int\limits_{\cF(t)} \Delta \ovd \cdot \ovd\left(|\ovd|^2-1\right)\rd x 
        &= 2\lambda\int\limits_{\cF(t)} \nabla \ovd\cdot \nabla \left[\ovd\left(|\ovd|^2-1\right)\right]\rd x \\
        &= 2\lambda\int\limits_{\cF(t)} |\nabla\ovd|^2 \left(|\ovd|^2-1\right)\rd x + 4\lambda\int\limits_{\cF(t)} |\ovd|^2|\nabla\ovd|^2\rd x.
    \end{aligned}
\end{equation}
\noindent{\textit{Fourth term of \eqref{d:mult}.}} We calculate
\begin{equation}\label{eq:T4}
    - 2\lambda\int\limits_{\cF(t)} |\nabla \ovd|^2 |\ovd|^2(|\ovd|^2-1)\rd x= - 2\lambda\int\limits_{\cF(t)} |\nabla \ovd|^2 \left(|\ovd|^2-1\right)(|\ovd|^2-1)\rd x - 2\lambda\int\limits_{\cF(t)} |\nabla \ovd|^2 (|\ovd|^2-1)\rd x.
\end{equation}
Adding the relations \eqref{eq:T1}--\eqref{eq:T4}, we deduce from the equation \eqref{d:mult} that
\begin{equation}\label{eq:T5}
    \frac{\mathrm{d}}{\mathrm{d}t}\int\limits_{\cF(t)} \gamma\left(\frac{1}{2}\left(|\ovd|^2|\ovd|^2\right) - |\ovd|^2 +\frac{1}{2}\right)\rd x + 4\lambda\int\limits_{\cF(t)} |\ovd|^2|\nabla\ovd|^2\rd x - 2\lambda\int\limits_{\cF(t)} |\nabla \ovd|^2 ||\ovd|^2-1|^2\rd x=0.
\end{equation}
Setting $\varphi:= |\ovd|^2-1$, we derive that
\begin{equation}\label{eq:T6}
|\varphi|^2= |\ovd|^2|\ovd|^2 - 2|\ovd|^2 +{1}, \tand |\nabla\varphi|^2=4|\ovd|^2|\nabla \ovd|^2.
\end{equation}
By virtue of \eqref{eq:T6}, the equation \eqref{eq:T5} can be rewritten as 
\begin{equation*}
    \frac{\gamma}{2}\frac{\mathrm{d}}{\mathrm{d}t}\int\limits_{\cF(t)} |\varphi|^2\rd x + \lambda\int\limits_{\cF(t)} |\nabla\varphi|^2\rd x = 2\lambda\int\limits_{\cF(t)} |\nabla \ovd|^2 |\varphi|^2\rd x.
\end{equation*}
As $\ovd\in \rW^{1,p}\left(0,T;\rL^q(\cF(\cdot))^3\right) \cap \rL^p\left(0,T;\rW^{2,q}(\cF(\cdot))^3\right)$ and $\gamma,\lambda >0$, for
\begin{equation*}
    \zeta(t):=\int\limits_{\cF(t)} |\varphi (t,x)|^2\rd x,
\end{equation*}
we find the inequality $\zeta'(t)\le C\zeta(t)$ for some $C > 0$.
As $|d_0|=1$, we obtain $\zeta(0)=0$, so Gronwall's inequality yields $\zeta=0$. 
This implies $\varphi=0$, and we conclude $\left|\ovd(t,x)\right|\equiv 1$ for all $(t,x)\in \cF_T$.
\end{proof}

\subsection{Energy and equilibria of the interaction problem}\label{ssec:energy}
\

Let us consider the energy of the interaction system \eqref{eq:LC-fluid equationsiso}--\eqref{eq:initial} defined by
\begin{equation}\label{energy}
    \begin{aligned}
        \rE(t)&\coloneqq \rE_{\mathrm{kin}}(t) + \rE_{\mathrm{pot}}(t)+\rE_{\mathrm{trans}}(t)+\rE_{\mathrm{rot}}(t) \\ &\coloneqq \frac{1}{2}\int\limits_{\cF(t)} \left|u^{\cF}(t,x)\right|^2\rd x +  \frac{1}{2}\int\limits_{\cF(t)} \left|\nabla \ovd (t,x)\right|^2\rd x+ \frac{1}{2} \mS |h'(t)|^2+ \frac{1}{2} (J\Omega\cdot\Omega)(t).
\end{aligned}
\end{equation}
We now show that $\rE(t)$ is non-increasing and determine the set of equilibria $\cE$ of \eqref{eq:LC-fluid equationsiso}--\eqref{eq:initial}.

\begin{prop}\label{prop:shape of equilibria}
Let $p,q \in (1,\infty)$ satisfy \eqref{eq:cond p and q} and $z_0 = (v_0,d_0,\ell_0,\omega_0) \in \rX_\gamma$, with $|d_0| = 1$ and $\rX_\gamma$ as defined in \eqref{eq:trace space}.
Moreover, let $(u^{\cF},\pi^{\cF},\ovd,h',\Omega)$ be a solution to \eqref{eq:LC-fluid equationsiso}--\eqref{eq:initial} satisfying
\begin{equation*}
    \begin{aligned}
        (u^{\cF},\pi^{\cF},\ovd,h',\Omega) &
        \in \rW^{1,p}\left(0,T;\rL^q(\cF(\cdot))^3\right) \cap \rL^p\left(0,T;\rW^{2,q}(\cF(\cdot))^3\right) \times \rL^p\left(0,T;\rW^{1,q}(\cF(\cdot)) \cap \rL_0^q(\cF(\cdot))\right)\\
        &\quad \times \rW^{1,p}\left(0,T;\rL^q(\cF(\cdot))^3\right) \cap \rL^p\left(0,T;\rW^{2,q}(\cF(\cdot))^3\right) \times \rW^{1,p}(0,T)^3
         \times \rW^{1,p}(0,T)^3.
    \end{aligned}
\end{equation*}
The energy functional $\rE$ defined in \eqref{energy} is non-increasing. 
Moreover, the set of equilibria is given by
\begin{equation*}
    \cE=\left\{(0,d_*,0,0) : \Delta d_*+|\nabla d_*|^2d_*=0 \tin \cF_0, \enspace \partial_{\nu}d_*=0 \ton \partial\cF_0\right\}.
\end{equation*}
\end{prop}

\begin{proof}
Let us first prove that $\rE(t)$ is non-increasing along the solutions. 
We use integration by parts, the Reynold's transport theorem, the equations \eqref{eq:LC-fluid equationsiso}$_1$--\eqref{eq:LC-fluid equationsiso}$_2$, and the boundary conditions \eqref{eq:bdry cond} to derive
\begin{equation}\label{Ekin}
    \begin{aligned}
    \frac{\mathrm{d}}{\mathrm{d}t}\rE_{\mathrm{kin}}(t) 
    &=   \int\limits_{\cF(t)} \partial_t u^{\cF}\cdot u^{\cF}\rd x+ \frac{1}{2}\int\limits_{\partial\cS(t)} |u^{\cF}|^2 \left(u^{\cF}\cdot n\right)\rd \Gamma \\ 
    &= -\int\limits_{\cF(t)} \left(u^{\cF}\cdot\nabla\right)u^{\cF}\cdot u^{\cF}\rd x + \int\limits_{\cF(t)}\mdiv\Sigma(u^{\cF},\pi^{\cF}, \ovd)\cdot u^{\cF}\rd x + \frac{1}{2}\int\limits_{\partial\cS(t)} |u^{\cF}|^2 \left(u^{\cF}\cdot n\right)\rd \Gamma\\ 
    &= -\int\limits_{\cF(t)} |\nabla u^{\cF}|^2\rd x + \int\limits_{\cF(t)} \nabla \ovd (\nabla\ovd)^{\top}:\nabla u^{\cF} \rd x - \int\limits_{\partial\cS(t)} \Sigma(u^{\cF},\pi^{\cF}, \ovd)n\cdot u^{\cF} \rd \Gamma.
    \end{aligned}
\end{equation}
Another integration by parts in conjunction with the Reynold's transport theorem leads to
\begin{equation}\label{Epot1}
    \begin{aligned}
        \frac{\mathrm{d}}{\mathrm{d}t}\rE_{\mathrm{pot}}(t) = \frac{1}{2}\int\limits_{\cF(t)} \partial_t |\nabla\ovd|^2\rd x + \frac{1}{2} \int\limits_{\partial\cS(t)} |\nabla\ovd|^2 \left(u^{\cF}\cdot n\right)\rd \Gamma.
  \end{aligned}
\end{equation}
By using \eqref{eq:LC-fluid equationsiso}$_3$ and $|\ovd|=1$, following from \autoref{prop:preservation of the dir cond}, we obtain the identity
\begin{equation}\label{Epot2}
    \begin{aligned}
        &\quad \int\limits_{\cF(t)} |\Delta \ovd+|\nabla \ovd|^2\ovd|^2\rd x = \int\limits_{\cF(t)} \left(\Delta \ovd+|\nabla \ovd|^2\ovd\right)\cdot \left(\partial_t\ovd + \left(u^{\cF}\cdot\nabla\right)\ovd\right)\rd x = -\int\limits_{\cF(t)} \partial_t(\nabla\ovd):\nabla\ovd \rd x\\
        &\quad + \frac{1}{2}\int\limits_{\cF(t)} |\nabla \ovd|^2 \partial_t |\ovd|^2\rd x + \int\limits_{\cF(t)} |\nabla\ovd|^2 \left(u^{\cF}\cdot\nabla\right) \frac{|\ovd|^2}{2} \rd x + \int\limits_{\cF(t)}  \Delta \ovd \cdot \left(u^{\cF}\cdot\nabla\right)\ovd \rd x \\
        &= -\frac{1}{2}\int\limits_{\cF(t)} \partial_t |\nabla\ovd|^2\rd x - \int\limits_{\cF(t)}  \nabla \ovd : \nabla \left[\left(u^{\cF}\cdot\nabla\right)\ovd \right]\rd x.
    \end{aligned}
\end{equation}
Furthermore, we use integration by parts as well as the relation $\mdiv u^{\cF}=0$ in $\cF(t)$ to get
\begin{equation}\label{Epot3}
    \begin{aligned}    
        - \int\limits_{\cF(t)}  \nabla \ovd : \nabla \left[\left(u^{\cF}\cdot\nabla\right)\ovd \right]\rd x
        &= - \sum\limits_{i,j,k=1}^3 \int\limits_{\cF(t)}\partial_k \left(u^{\cF}_i\partial_i \ovd_j\right)\partial_k\ovd_j \rd x\\
        &= - \sum\limits_{i,j,k=1}^3 \int\limits_{\cF(t)}\left[\partial_k u^{\cF}_i\partial_i\ovd_j\partial_k\ovd_j+u^{\cF}_i\partial_i\partial_k \ovd_j\partial_k \ovd_j\right]\rd x \\
        &=-\int\limits_{\cF(t)} \nabla \ovd (\nabla\ovd)^{\top}:\nabla u^{\cF} \rd x - \frac{1}{2} \int\limits_{\partial\cS(t)} |\nabla\ovd|^2 \left(u^{\cF}\cdot n\right)\rd \Gamma.
    \end{aligned}
\end{equation}
Hence, combining \eqref{Epot1}--\eqref{Epot3}, we infer that
\begin{equation}\label{Epot}
    \frac{\mathrm{d}}{\mathrm{d}t}\rE_{\mathrm{pot}}(t) = -\int\limits_{\cF(t)} |\Delta \ovd+|\nabla \ovd|^2|\ovd|^2\rd x -\int\limits_{\cF(t)} \nabla \ovd (\nabla\ovd)^{\top}:\nabla u^{\cF} \rd x.
\end{equation}
The interface condition $u^{\cF}=h'(t)+\Omega(t)\times (x-h(t))$ on $\partial\cS(t)$ and the rigid body equations \eqref{eq:rigid body eqs intro} yield
\begin{equation}\label{surface}
    \begin{aligned}
        &\quad -\int\limits_{\partial\cS(t)} \Sigma(u^{\cF},\pi^{\cF}, \ovd)n\cdot u^{\cF} \rd \Gamma\\
        &= -h'(t)\cdot \int\limits_{\partial\cS(t)} \Sigma(u^{\cF},\pi^{\cF}, \ovd)n \rd \Gamma - \Omega(t)\cdot \int\limits_{\partial\cS(t)} (x-h(t))\times \Sigma(u^{\cF},\pi^{\cF}, \ovd)n \rd \Gamma\\
        &=-h'(t)\cdot \mS h''(t) - \Omega(t)\cdot (J\Omega)'(t)= -\frac{\mathrm{d}}{\mathrm{d}t}\rE_{\mathrm{trans}}(t)- \Omega(t)\cdot (J\Omega)'(t).
    \end{aligned}
\end{equation}
From \eqref{eq:Syl and Jab}, it follows that
\begin{equation}\label{Jab1}
    J(t)\Omega(t)'\cdot\Omega(t) = J(t)\Omega(t)\cdot\Omega'(t), \tforall t > 0.
\end{equation}
Using Sylvester's law from \eqref{eq:Syl and Jab} and \eqref{Jab1}, we obtain
\begin{equation}\label{Erot}
    \begin{aligned}
        \frac{\mathrm{d}}{\mathrm{d}t}\rE_{\mathrm{rot}}(t) &= \frac{1}{2}\left(\Omega(t)\cdot (J\Omega)'(t) + \Omega'(t)\cdot (J\Omega)(t)\right)\\
        &= \frac{1}{2}\left[\Omega(t)\cdot (J\Omega'+ \Omega\times J\Omega)(t) + \Omega'(t)\cdot (J\Omega)(t)\right] = \Omega(t)\cdot (J\Omega)'(t).
  \end{aligned}
\end{equation}
Therefore, we concatenate \eqref{Ekin}, \eqref{Epot}, \eqref{surface} and \eqref{Erot} to conclude
\begin{equation}\label{Efinal}
    \frac{\mathrm{d}}{\mathrm{d}t}\rE(t)= -\int\limits_{\cF(t)} |\nabla u^{\cF}|^2\rd x-\int\limits_{\cF(t)} |\Delta \ovd+|\nabla \ovd|^2\ovd|^2\rd x.
\end{equation}
The equilibria $\cE$ are now precisely the critical points of $\rE$, so suppose that $\frac{\mathrm{d}}{\mathrm{d}t}\rE(t) \big|_{t=t_0}=0$ at $t_0 > 0$.
Then \eqref{Efinal} and $u^{\cF}=0$ on $ \partial\cO$ imply $u^{\cF}=0$ in $\cF(t_0)$, and $\ovd$ satisfies $\Delta \ovd+|\nabla \ovd|^2\ovd=0$ in $\cF(t_0)$. 
Moreover, from $u^{\cF}=h'(t)+\Omega(t)\times (x-h(t))$ on $\partial\cS(t)$ and $u^{\cF}=0$ in $\cF(t)$, we conclude by a slight adjustment of \cite[Lemma~2.1]{RT:19} to the 3D case that $h'=0$ and $\Omega=0$, and $\cE$ has the shape as asserted.
\end{proof}

\section{Change of variables}
\label{sec:change of var}

The transformation to a fixed domain presented in the sequel is classical, see for instance the work of Inoue and Wakimoto \cite{IW:77}.
However, we provide it here for convenience.
To this end, let us denote by $m(t)$ the skew-symmetric matrix satisfying $m(t) x = \Omega(t) \times x$.
We then consider the initial value problem
\begin{equation}\label{eq:IVP X0}
\left\{
	\begin{aligned}
		\partial_t X_0(t,y) &= m(t)(X_0(t,y) - h(t)) + h'(t), &&\ton (0,T) \times \R^3, \\
		X_0(0,y) &= y, &&\tfor y \in \R^3. 
	\end{aligned}
\right. 	
\end{equation}
The solution of \eqref{eq:IVP X0} is given by $X_0(t,y) = Q(t) y + h(t)$ for $Q(t) \in \SO(3)$, where $Q \in \rW^{2,p}(0,T)^{3 \times 3}$ is valid provided $h'$, $\Omega \in \rW^{1,p}(0,T)^3$.
The inverse $Y_0(t,x) = Q^{\top}(t)(x - h(t))$ of $X_0(t)$ solves
\begin{equation*}
\left\{
	\begin{aligned}
		\partial_t Y_0(t,x) &= -M(t) Y_0(t,x) - \ell(t), &&\ton (0,T) \times \R^3, \\
		Y_0(0,x) &= x, &&\tfor y \in \R^3, 
	\end{aligned}
\right. 	
\end{equation*}
where $M(t) \coloneqq Q^{\top}(t) m(t) Q(t)$ and $\ell(t) \coloneqq Q^{\top}(t) h'(t)$.
Next, we adjust the diffeomorphisms $X_0$ and $Y_0$ of $\cF(t)$ as well as $\cS(t)$ such that the transform only acts in a suitable open neighborhood of the moving rigid body.
To this end, we define a new diffeomorphism $X$ implicitly by demanding that it solves the initial value problem
\begin{equation}\label{eq:IVP X}
\left\{
	\begin{aligned}
		\partial_t X(t,y) &= b(t,X(t,y)), &&\ton (0,T) \times \R^3, \\
		X(0,y) &= y, &&\tfor y \in \R^3.
	\end{aligned}
\right. 	
\end{equation}
The right-hand side $b$ will be constructed such that the transform fulfills the above requirements.
We assume that the rigid body starts from a position with strictly positive distance from the outer boundary, i.e., $\dist(\cS_0,\partial \cO) > r$ for some $r > 0$.
As the body velocity is continuous, we only consider the solution up to a time for which a small distance such as $\nicefrac{r}{2}$ of the rigid body and the outer boundary is preserved.
Accordingly, we define a cut-off function $\chi \in \rC^{\infty}(\R^3;[0,1])$ by
\begin{equation*}
\chi(x) \coloneqq \left\{
	\begin{aligned}
		1, &\tif \dist(x,\partial \cO) \ge r, \\
		0, &\tif \dist(x,\partial \cO) \le \nicefrac{r}{2}.
	\end{aligned}
\right. 	
\end{equation*}
Equipped with this cut-off function, we define the right-hand side $b \colon [0,T] \times \R^3 \to \R^3$ of \eqref{eq:IVP X} by
\begin{equation}\label{eq:rhs b}
    b(t,x) \coloneqq \chi(x)[m(t)(x - h(t)) + h'(t)] - B_{\cF_0}(\nabla \chi(\cdot)[m(t)(\cdot - h(t)) + h'(t)])(x),
\end{equation}
where $B_{\cF_0} \colon \rC_c^\infty(\cF_0) \to \rC_c^\infty(\cF_0)^3$ represents the Bogovski\u{\i} operator associated to $\cF_0$.
This operator is bounded and yields $\mathrm{div} B_{\cF_0} g = g$ provided $\int_{\cF_0} g \rd y = 0$.
We also refer to \cite{Bog:79, Gal:94, GHH:06} for further details on the Bogovski\u{\i} operator.
In our situation, it readily follows that $\nabla \chi(x)[m(t)(x - h(t)) + h'(t)]$ has mean value zero in $\cF_0$, so the correction by the Bogovski\u{\i} results in $\mdiv(b(t,\cdot)) = 0$ for all $t \in [0,T]$.
Consequently, $b \in \rW^{1,p}\left(0,T;\rC_{c,\sigma}(\cF_0)^3\right)$ as well as $b|_{\del \cS_0} = m(x-h) + h'$.

The Picard-Lindel\"of theorem yields the existence of a unique solution $X \in \rC^1\left((0,T);\rC^\infty(\R^3)^3\right)$ to \eqref{eq:IVP X} for given $h'$, $\Omega \in \rW^{1,p}(0,T)^3$.
We also observe that the solution has continuous mixed partial derivatives $ \frac{\partial^{|\alpha|+1} X}{\partial t (\partial y_j)^{\alpha}}$ and $\frac{\partial^{|\alpha|} X}{(\partial y_j)^{\alpha}}$ for a multi-index $\alpha \in \N_0^3$.
The uniqueness of the solution implies that the function $X(t,\cdot)$ is bijective, with inverse denoted by $Y(t,\cdot)$.
Thanks to $\mdiv b = 0$, Liouville's theorem yields that $X$ and $Y$ are both volume-preserving, so $\rJ_X(t,y) \rJ_Y(t,X(t,y)) = \Id$ and $\det \rJ_X(t,y) = \det \rJ_Y(t,x) = 1$, where, as usual, $\rJ_X$ and $\rJ_Y$ represent the Jacobian matrices of $X$ and $Y$.
Let us observe that the inverse transform $Y$ of $X$ solves the initial value problem
\begin{equation}\label{eq:IVP Y}
\left\{
	\begin{aligned}
		\partial_t Y(t,x) &= b^{(Y)}(t,Y(t,x)), &&\ton (0,T) \times \R^3, \\
		Y(0,x) &= x, &&\tfor x \in \R^3,
	\end{aligned}
\right. 	
\end{equation}
with right-hand side given by
\begin{equation}\label{eq:rhs bY}
    b^{(Y)}(t,y) \coloneqq -\rJ_X^{-1}(t,y) b(t,X(t,y)).
\end{equation}
We emphasize that $b^{(Y)}$ as well as $Y$ have the same space and time regularity as $b$ and $X$, and that the interval of existence of the solution is restricted by the condition $\dist(\cS(t),\partial \cO) > \nicefrac{r}{2}$ for all $t \in (0,T)$.

As it does not affect the analysis, we will assume that $\gamma$, $\lambda$ and $\mu$ are equal to one in the remainder of this article.
Furthermore, as a preparation for the proof of the global well-posedness, we subtract an arbitrary constant vector $d_* \in \R^3$ with $|d_*| = 1$ in the $d$-variable, which does not change the analysis of the local strong well-posedness significantly.
For $X \colon\R^3\rightarrow\R^3$ as in \eqref{eq:IVP X} and $(t,y)\in (0,T) \times \R^3$, we define
\begin{equation*}
\left\{
    \begin{aligned}
        v(t,y) &\coloneqq J_{Y}(t,X(t,y))u^{\cF}(t,X(t,y)), \enspace p(t,y) \coloneqq\pi^{\cF}(t,X(t,y)),\\ 
        d (t,y) &\coloneqq \ovd(t,X(t,y)) - d_*,\\
        \ell(t) &\coloneqq Q(t)^{\top}h'(t), \tand \omega(t) \coloneqq Q(t)^{\top}\Omega(t).
    \end{aligned}
\right.
\end{equation*}
Moreover, we set the transformed stress tensor as
\begin{equation*}
\sigma(v(t,y),p(t,y), d(t,y)) \coloneqq Q(t)^{\top}\Sigma(Q(t)v(t,y),p(t,y),d(t,y))Q(t).
\end{equation*}
The metric contravariant $(g^{ij})$ and covariant tensors $(g_{ij})$ and the Christoffel symbol are given by
\begin{equation*}
    g^{ij}= \sum_{k=1}^3 (\partial_k Y_i)(\partial_k Y_j), \enspace g_{ij}= \sum_{k=1}^3 (\partial_i X_k)(\partial_j X_k), \tand \Gamma^{i}_{jk}= \frac{1}{2} \sum_{l=1}^3 g^{il} (\partial_k g_{jl} + \partial_j g_{kl} - \partial_l g_{jk}).
\end{equation*}
Having these terms at hand, we introduce the transformed Laplacian operators $\cL_1$ and $\cL_2$ corresponding to the velocity and the director field, respectively.
They take the shape
\begin{equation*}
    \begin{aligned}
        (\cL_1 v)_{i} 
        &= \sum_{j,k=1}^{3} \partial_j (g^{jk}\partial_k v_i) + 2 \sum_{j,k,l=1}^{3}  g^{kl}\Gamma_{jk}^i\partial_l v_j +  \sum_{j,k,l=1}^{3}\left(\partial_k(g^{kl}\Gamma^{i}_{kl}) + \sum_{m=1}^3 g^{kl}\Gamma^{m}_{jl}\Gamma^{i}_{km}\right)v_j, \tand\\
        (\cL_2 d)_{i} 
        &= \sum\limits_{j,k=1}^3 \left((\partial_{k}\partial_{j}d_{i})g^{jk}\right) + \sum\limits_{k=1}^3 (\Delta Y_k) (\partial_k d_i).
    \end{aligned}
\end{equation*}
For the representation of $\cL_1$, we refer e.\ g.\ to \cite[(3.13)]{GGH:13}, while $\cL_2$ can for instance be found in \cite[(3.13)]{BBH:23}.
Defining $B(d)h$ with $B(d)d=\mdiv(\nabla d (\nabla d)^{\top})$ by 
\begin{equation}\label{eq:def B}
    [B(d)h]_{i} \coloneqq \partial_i d_l \Delta h_l + \partial_k d_l \partial_k \partial_i h_l,
\end{equation}
we calculate the transformed terms associated to the gradient of the pressure, to $B(d)h$, to the time derivative and to the convective term in the fluid equation
\begin{equation*}
    \begin{aligned}
        (\cG p)_{i} = \sum\limits_{j=1}^3 g^{ij}\partial_j p, \enspace (\cB(d) h)_{i}
        &= \sum\limits_{l=1}^3 \left((\cL_2 h)_{l}\sum\limits_{m=1}^3 {\partial_m d_{l}}{\partial_i Y_m}  \right)\\
        &\quad + \sum\limits_{k,l,m=1}^3 \left( {\partial_m d_{l}}{\partial_k Y_m}\right)\left(\sum\limits_{j=1}^3 (\partial_j \partial_m h_{l})(\partial_k Y_j \partial_i Y_m) +  (\partial_m h_l )(\partial_k \partial_i Y_m )\right),
    \end{aligned}
\end{equation*}
\begin{equation*}
    (\cM v)_{i} = \sum\limits_{j=1}^3 \partial_t Y_j \partial_j v_i + \sum\limits_{j,k=1}^3 \left(\Gamma^i_{jk}\partial_t Y_k + (\partial_k Y_i)(\partial_j \partial_t X_k)\right)v_j, \tand (\cN(v))_{i} = \sum\limits_{j=1}^3 v_j\partial_j v_i + \sum\limits_{j,k=1}^3 (\Gamma^i_{jk} v_j v_k).
\end{equation*}

It is then possible to rewrite the system \eqref{eq:LC-fluid equationsiso}--\eqref{eq:initial} on the cylindrical domain $(0,T)\times\cF_0$ as
\begin{equation}\label{eq:cv1}
\left\{
    \begin{aligned}
        \partial_t v + (\cM-\cL_1)v +\cN(v)+\cG p+\cB(d)d
        &=0, &&\tin (0,T) \times \cF_0,\\
        \mdiv v
        &=0, &&\tin (0,T) \times \cF_0,\\
        {\partial_t d} + \nabla d\cdot \partial_t Y + (v\cdot \nabla) d - \cL_2 d  
        &=  |\nabla d|^2 (d+d_*), &&\tin (0,T) \times \cF_0,\\
        \mS\ell' + \mS(\omega\times\ell)
        &= -\int\limits_{\partial \cS_0} \sigma(v,p, d) N\rd \Gamma, &&\tin (0,T),\\
        J_0\omega' - \omega\times (J_0\omega) 
        &=-\int\limits_{\partial \cS_0} y \times  \sigma(v,p, d) N\rd \Gamma, &&\tin (0,T),
    \end{aligned}
\right.
\end{equation}
where $N = Q^\top(t) \nu(t)$ denotes the unit outward normal to the boundary of $\cF_0$, so it is directed inwards to $\cS_0$, and with the boundary conditions
\begin{equation}\label{eq:cv2}
    v = 0, \enspace \partial_\nu d = 0 \ton (0,T) \times \partial \cO, \tand v = \ell+\omega\times y, \enspace \partial_\nu d = 0 \ton \partial\cS_0,
\end{equation}
and the initial conditions, which are slightly adjusted for the director part by the transform,
\begin{equation}\label{eq:cv3}
    v(0)=v_0, \enspace d(0) = d_0 - d_*, \enspace \ell(0) = \ell_0 \tand \omega(0) = \omega_0.
\end{equation}
We can now reformulate the first main result, \autoref{thm:local wp}, on the reference configuration.

\begin{thm}\label{thm:local wp reformulated}
Let $p,q \in (1,\infty)$ be such that \eqref{eq:cond p and q} holds true, let $\cO \subset \R^3$ be a bounded domain of class $\rC^3$, and consider the domains $\cS_0$ and $\cF_0$ occupied by the rigid body and the fluid at initial time.
Furthermore, let $z_0 = (v_0,d_0,\ell_0,\omega_0) \in \rX_\gamma$, with $\rX_\gamma$ as defined in \eqref{eq:trace space}, and $|d_0| = 1$.
If for some $r > 0$, it is valid that $\dist(\cS_0,\partial \cO) > r$, then there are $T' \in (0,T]$ and $X \in \rC^1\left([0,T'];\rC^2(\R^3)^3\right)$ such that $X(\tau,\cdot) \colon \cF(\tau) \to \cF_0$ are $\rC^2$-diffeomorphisms for all $\tau \in [0,T']$, and the transformed interaction problem \eqref{eq:cv1}--\eqref{eq:cv3} admits a unique solution
\begin{equation*}
    \begin{aligned}
        v &\in \rW^{1,p}\left(0,T';\rL^q(\cF_0)^3\right) \cap \rL^p\left(0,T';\rW^{2,q}(\cF_0)^3\right), \enspace p \in \rL^p\left(0,T';\rW^{1,q}(\cF_0) \cap \rL_0^q(\cF_0)\right),\\
        d &\in \rW^{1,p}\left(0,T';\rL^q(\cF_0)^3\right) \cap \rL^p\left(0,T';\rW^{2,q}(\cF_0)^3\right),\\
        \ell &\in \rW^{1,p}(0,T')^3, \tand \omega \in \rW^{1,p}(0,T')^3.
    \end{aligned}
\end{equation*} 
\end{thm}

Next, we rewrite the second main result, \autoref{thm:global wp close to equilibria}, on the fixed domain.

\begin{thm}\label{thm:global wp close to equilibria reformulated}
Let $p,q \in (1,\infty)$ be such that \eqref{eq:cond p and q} holds true, let $\cO \subset \R^3$ be a bounded domain of class~$\rC^3$, and consider the domains $\cS_0$ and $\cF_0$ occupied by the rigid body and the fluid at time zero.
Moreover, assume that $\dist(\cS_0,\partial \cO) > r$ for some $r > 0$, and let $d_* \in \R^3$ constant with $|d_*| = 1$.
Then there is $\eta_0 > 0$ such that for all $\eta \in (0,\eta_0)$, there are two constants $\delta_0$ and $C > 0$ such that for all $\delta \in (0,\delta_0)$ and for all $z_0 = (v_0,d_0,\ell_0,\omega_0) \in \rX_\gamma$, with $\rX_\gamma$ as in \eqref{eq:trace space}, $|d_0| = 1$ and $\| (v_0, d_0 - d_*, \ell_0,\omega_0) \|_{\rX_\gamma} \le \delta$, then $X \in \rC^1\left([0,\infty);\rC^2(\R^3)^3\right)$ is such that $X(\tau,\cdot) \colon \cF(\tau) \to \cF_0$ are $\rC^2$-diffeomorphisms for all $\tau \in [0,\infty)$, and the transformed interaction problem \eqref{eq:cv1}--\eqref{eq:cv3} admits a unique strong solution $(v,p,d,\ell,\omega)$ with
\begin{equation*}
    \begin{aligned}
        &\quad \left\| \mre^{\eta(\cdot)} v \right\|_{\rL^p(0,\infty;\rW^{2,q}(\cF_0))} + \left\| \mre^{\eta(\cdot)} \partial_t v \right\|_{\rL^p(0,\infty;\rL^q(\cF_0))} + \left\| \mre^{\eta(\cdot)} v \right\|_{\rL^\infty\left(0,\infty;\rB_{qp}^{2-\nicefrac{2}{p}}(\cF_0)\right)}\\
        & + \left\| \mre^{\eta(\cdot)} p \right\|_{\rL^p(0,\infty;\rW^{1,q}(\cF_0) \cap \rL_0^q(\cF_0))} + \left\| \mre^{\eta(\cdot)} \nabla d \right\|_{\rL^p(0,\infty;\rL^q(\cF_0))} + \left\| \mre^{\eta(\cdot)} \nabla^2 d \right\|_{\rL^p(0,\infty;\rL^q(\cF_0))}\\
        & + \left\| \mre^{\eta(\cdot)} \partial_t d \right\|_{\rL^p(0,\infty;\rL^q(\cF_0))} + \left\| d \right\|_{\rL^\infty\left(0,\infty;\rB_{qp}^{2-\nicefrac{2}{p}}(\cF_0)\right)} + \left\| \mre^{\eta(\cdot)} \ell \right\|_{\rW^{1,p}(0,\infty)} + \left\| \mre^{\eta(\cdot)} \omega \right\|_{\rW^{1,p}(0,\infty)} \le C \delta.
    \end{aligned}
\end{equation*}
\end{thm}

\section{Linear theory}
\label{sec:lin theory}

In this section, we establish maximal regularity of the linearized interaction problem for the fixed point procedure of the local strong well-posedness.With regard to the global strong well-posedness, we show that the linearization at a constant equilibrium admits maximal $\rL^p([0,\infty))$-regularity up to a shift.

In the sequel, we introduce some further spaces to shorten notation.
For this purpose, let $0 < T \le \infty$ as well as $p$, $q \in (1,\infty)$.
We introduce the space for the terms on the right-hand side, and it is given by
\begin{equation}\label{eq:data space}
    \F_T \coloneqq \F_T^v \times \F_T^d \times \F_T^\ell \times \F_T^\omega \coloneqq \rL^p\left(0,T;\rL^q(\cF_0)^3\right) \times \rL^p\left(0,T;\rL^q(\cF_0)^3\right) \times \rL^p(0,T)^3 \times \rL^p(0,T)^3.
\end{equation}
Furthermore, for $\rL_\sigma^q(\cF_0)$ denoting the divergence free vector fields in $\rL^q(\cF_0)^3$ and $z = (v,d,\ell,\omega)$, we set
\begin{equation}\label{eq:ground and regularity space}
    \begin{aligned}
        \rX_0
        \coloneqq &\rL_\sigma^q(\cF_0) \times \rL^q(\cF_0)^3 \times \R^3 \times \R^3, \tand\\
        \rX_1
        \coloneqq &\bigl\{
        z \in \rW^{2,q}(\cF_0)^3 \cap \rL_\sigma^q(\cF_0) \times \rW^{2,q}(\cF_0)^3 \times \R^3 \times \R^3 : v = 0, \enspace \del_\nu d = 0, \ton \del \cO, \tand\\
        &\enspace v = \ell + \omega \times y, \enspace \del_\nu d = 0, \ton \del \cS_0\bigr\}.
    \end{aligned}
\end{equation}
Similarly as in \cite[Lemma~5.3]{BBH:23}, $\rX_\gamma$ can be obtained from $\rX_0$ and $\rX_1$ by real interpolation.
In order to simplify the notation, for $j \in \{v,d\}$, we will also use $\| \cdot \|_{\rX_0^j}$, $\| \cdot \|_{\rX_1^j}$ and $\| \cdot \|_{\rX_\gamma^j}$ to denote $\| \cdot \|_{\rL^q(\cF_0)}$, $\| \cdot \|_{\rW^{2,q}(\cF_0)}$ and $\| \cdot \|_{\rB_{qp}^{2-\nicefrac{2}{p}}(\cF_0)}$, respectively, where we remark that the norms are not affected by the boundary conditions, or by the consideration of closed subspaces.
Given $\rX_0$ and $\rX_1$, we further define the solution space $\E_T$ by
\begin{equation}\label{eq:sol space}
    \E_T \coloneqq \rW^{1,p}(0,T;\rX_0) \cap \rL^p(0,T;\rX_1).
\end{equation}
We employ the notation $\prescript{}{0}{\E_T}$ for the situation of homogeneous initial values.
As above, we use $\| \cdot \|_{\E_T^j}$ to denote $\| \cdot \|_{\rW^{1,p}(0,T;\rL^q(\cF_0)) \cap \rL^p(0,T;\rW^{2,q}(\cF_0))}$ for $j \in \{v,d\}$.
In order to account for the pressure, we set
\begin{equation}\label{eq:adjusted max reg space}
    \tE_T \coloneqq \E_T \times \E_T^p \coloneqq \E_T \times \rL^p(0,T;\rW^{1,q}(\cF_0) \cap \rL_0^q(\cF_0)),
\end{equation}
where the intersection with $\rL_0^q(\cF_0))$ is chosen to get uniqueness of the pressure.

Concerning the linearization, we do not include the part of the stress tensor linked to the director $d$, but we will treat this term as a nonlinear term on the right-hand side.
In that respect, we define
\begin{equation}\label{eq:lin stress tensor}
    \rT(v,p) \coloneqq 2 \D  v - p \Id. 
\end{equation}
For $d^* \in \BUC([0,T];\rX_\gamma^d)$, where $\BUC([0,T])$ denotes the bounded and uniformly continuous functions on $[0,T]$, determined precisely later and $B(d^*) d$ as in \eqref{eq:def B} as well as $d_* \in \R^3$ constant such that $|d_*| = 1$, the linearization reads as
\begin{equation}\label{eq:linearization loc}
\left\{
    \begin{aligned}
        \partial_t {v} -\Delta {v} +\nabla {p} + B(d^*) d &=f_1, \enspace \mdiv {v} = 0, \tand {\partial_t {d}}  - \Delta {d} = f_2, &&\tin (0,T) \times \cF_0,\\
        \mS({\ell})' + \int_{\del \cS_0} \rT(v,p) N \rd \Gamma &= f_3, &&\tin (0,T),\\ 
        J_0({\omega})' + \int_{\del \cS_0} y \times \rT(v,p) N \rd \Gamma &= f_4, &&\tin (0,T),\\
        {v}(t,y)&={\ell}(t)+{\omega}(t)\times y, &&\ton (0,T) \times \partial\cS_0,\\
        {v}(t,y)&=0, &&\ton (0,T) \times \partial \cO,\\
        \partial_\nu d &= 0, &&\ton (0,T) \times \partial \cF_0,\\
        v(0) = v_0, \enspace d(0) &= d_0 - d_*, \enspace \ell(0) = \ell_0, \tand \omega(0) = \omega_0, 
    \end{aligned}
\right.
\end{equation}
where $(f_1,f_2,f_3,f_4) \in \F_T$ and $(v_0,d_0,\ell_0,\omega_0) \in \rX_\gamma$.
First, we discuss the existence of a so-called reference solution to capture the initial values.

\begin{prop}\label{prop:reference sol}
Let $p,q \in (1,\infty)$ satisfy \eqref{eq:cond p and q}, and consider $(v_0,d_0,\ell_0,\omega_0) \in \rX_\gamma$, where $\rX_\gamma$ was introduced in \eqref{eq:trace space}. 
Moreover, consider \eqref{eq:linearization loc} with
\begin{equation*}
    \partial_t {v} -\Delta {v} +\nabla {p} + B(d^*) d =f_1 \enspace \text{replaced by} \enspace \partial_t {v} -\Delta {v} +\nabla {p} = f_1
\end{equation*}
as well as $(f_1,f_2,f_3,f_4) = (0,0,0,0)$.
Then there is a unique solution $(z^*,p^*) = (v^*,d^*,\ell^*,\omega^*,p^*) \in \tE_T$ to the resulting system.
\end{prop}

\begin{proof}
Let us observe that the equation for $d$ is completely decoupled.
On the one hand, this allows us to use $\E_T^d$ to denote the $d$-part of the maximal regularity space.
On the other hand, we can then invoke the well known maximal regularity of the Neumann Laplacian to obtain $d^* \in \E_T^d$ satisfying the heat equation with homogeneous right-hand side and initial value $d_0 - d_* \in \rX_\gamma^d$.
For the remaining fluid-structure part, we refer for instance to \cite[Theorem~4.1]{GGH:13} and its analogue in the present setting, where we observe that $z_0 \in \rX_\gamma$ means precisely that the compatibility conditions $(CN)$ from \cite[p.~1398]{GGH:13} adjusted to our situation are satisfied.
\end{proof}

\begin{rem}\label{rem:norm of reference sol}
Let $(z^*,p^*) = (v^*,d^*,\ell^*,\omega^*,p^*) \in \tE_{T_0}$ be the reference solution on $(0,T_0)$ resulting from \autoref{prop:reference sol}.
The part $d^*$ is obtained by a convolution of the heat semigroup with the initial data $d_0 - d_*$, so its $\rW^{1,p}(0,T;\rL^q(\cF_0))\cap \rL^p(0,T;\rW^{2,q}(\cF_0))$-norm tends to zero as $T \to 0$, because the length of the time interval tends to zero.
For the fluid-structure part, it can also be shown that the norm of the reference solution shrinks to zero, see for example \cite[Lemma~6.6]{GGH:13}.
Hence, we get
\begin{equation*}
    C_T^* \coloneqq \| z^* \|_{\E_T} + \| p^* \|_{\E_T^{p}} \to 0 \tfor T \to 0.
\end{equation*}
\end{rem}

\begin{rem}
It would be interesting to find out whether the operator associated to the linearized fluid-structure problem admits a bounded $\mathcal{H}^\infty$-calculus.
For further information on the latter property, we also refer to \cite{CDIY:96}.
\end{rem}

The following well known result provides uniform-in-time estimates of the trace space norms of the principle variable and the $d$-part by the initial data and the norm of the solution, where the constant in the estimate is time-independent. 

\begin{lem}\label{lem:est of BUC trace by MR space and IVs}
Let $z \in \E_T$.
Then for a constant $C>0$ independent of $T$, it holds that
\begin{equation*}
    \sup_{t \in [0,T]} \| z(t) \|_{\rX_\gamma} \le C\left(\| z(0) \|_{\rX_\gamma} + \| z \|_{\E_T}\right), \tand \sup_{t \in [0,T]} \| d(t) \|_{\rX_\gamma^d} \le C\left(\| d(0) \|_{\rX_\gamma^d} + \| d \|_{\E_T^d}\right).
\end{equation*}
\end{lem}

Next, we prove the maximal regularity of the linearized problem \eqref{eq:linearization loc}.

\begin{prop}\label{prop:max reg hom IVs}
Let $0 < T_0 \le \infty$, $T \in (0,T_0)$, $p,q \in (1,\infty)$ such that \eqref{eq:cond p and q} holds, and consider the reference solution $d^*$ from \autoref{prop:reference sol}, $(f_1,f_2,f_3,f_4) \in \F_T$ and $(v_0,d_0 - d_*,\ell_0,\omega_0) = (0,0,0,0)$. 
Then there is a unique solution $(\tv,\td,\tell,\tomega,\tp) \in \prescript{}{0}{\tE_T}$ to \eqref{eq:linearization loc} such that for the maximal regularity constant $\Cmr = \Cmr(T_0) > 0$ independent of $T$, we have the estimate
\begin{equation*}
    \left\| (\tv,\td,\tell,\tomega,\tp) \right\|_{\tE_T} \le \Cmr \cdot \| (f_1,f_2,f_3,f_4) \|_{\F_T}.
\end{equation*}
\end{prop}

\begin{proof}
The proof relies on the upper triangular structure of the linearized interaction problem and the maximal regularity of the fluid-structure system as established in \cite[Theorem~4.1]{GGH:13} in conjunction with the maximal regularity of the Neumann Laplacian operator.
More precisely, thanks to $f_2 \in \F_T^d$ and $d(0) = 0$, the maximal regularity of the Neumann Laplacian operator first yields $\td \in \E_T^d$ such that $\| \td \|_{\E_T} \le C \| f_2 \|_{\F_T}$, and the constant $C > 0$ is $T$-independent in view of the homogeneous initial values.
Next, we can insert the resulting $\td$ into the fluid equation and define $\Tilde{f_1} \coloneqq f_1 - B(d^*) \td \in \F_T^{v}$.
It then remains to investigate the system given by
\begin{equation*}
\left\{
    \begin{aligned}
        \partial_t {v} -\Delta {v} +\nabla {p} &=\Tilde{f_1}, \tand \mdiv {v} = 0, &&\tin (0,T) \times \cF_0,\\
        \mS({\ell})' + \int_{\del \cS_0} \rT(v,p) N \rd \Gamma &= f_3, \tand J_0({\omega})' + \int_{\del \cS_0} y \times \rT(v,p) N \rd \Gamma = f_4, &&\tin (0,T),\\ 
        {v}&={\ell}+{\omega}\times y, &&\ton (0,T) \times \partial\cS_0,\\
        {v}&=0, &&\ton (0,T) \times \partial \cO,\\
        v(0) = 0, \enspace \ell(0) &= 0, \tand \omega(0) = 0.
    \end{aligned}
\right.
\end{equation*}
From \cite[Theorem~4.1]{GGH:13}, it then follows that this system admits a unique solution $(\tv,\tp,\tell,\tomega)$ such that
\begin{equation}\label{eq:max reg est fluid-structure part}
    \left\| \tv \right\|_{\E_T^v} + \left\| \tp \right\|_{\rL^p(0,T;\rW^{1,q}(\cF_0) \cap \rL_0^q(\cF_0))} + \left\|(\tell,\tomega) \right\|_{\rW^{1,p}(0,T)} \le C\left(\left\| \Tilde{f_1} \right\|_{\F_T^v} + \| f_3 \|_{\F_T^\ell} + \| f_4 \|_{\F_T^\omega}\right),
\end{equation}
where $C > 0$ is independent of $T$ thanks to the homogeneous initial values.
It remains to estimate $\Tilde{f_1}$ in $\F_T^v = \rL^p(0,T;\rL^q(\cF_0)^3)$.
From \autoref{lem:est of BUC trace by MR space and IVs}, the above maximal regularity estimate of the heat equation with zero initial data and the existence of the reference solution $d^*$ up to time $T_0$, we deduce the estimate
\begin{equation*}
    \begin{aligned}
        \| B(d^*) \td \|_{\F_T^{v}}
        &\le C \| d^* \|_{\BUC([0,T];\rX_\gamma^{d})} \| \td \|_{\E_T^{d}}\le C \left(\| d^* \|_{\E_T^{d}} + \| d_0 \|_{\rX_\gamma^{d}}\right) \| f_2 \|_{\F_T^{d}}\le C \left(\| d^* \|_{\E_{T_0}^{d}} + \| d_0 \|_{\rX_\gamma^{d}}\right) \| f_2 \|_{\F_T^{d}},
    \end{aligned}
\end{equation*}
and we observe that $\| d^* \|_{\E_{T_0}^{d}}$ does not depend on $T$.
The latter estimate and~\eqref{eq:max reg est fluid-structure part} as well as the above $\td \in \E_T^d$ finally yield the desired solution $(\tv,\td,\tell,\tomega,\tp) \in \prescript{}{0}{\tE_T}$ to \eqref{eq:linearization loc} with homogeneous initial values.
The maximal regularity estimate is valid with constant $\Cmr > 0$ independent of $T$ in view of the homogeneous initial values and the preceding estimate of $\| B(d^*) \td \|_{\F_T^{v}}$.
\end{proof}

After establishing the maximal regularity results with regard to the local well-posedness, we now concentrate on the situation of the global well-posedness.
As the latter also relies on a fixed point argument, we aim for maximal $\rL^p([0,\infty))$-regularity of the linearization around constant equilibria.
The latter notion of maximal $\rL^p([0,\infty))$-regularity is generally a strong property that requires the underlying operator to be invertible, see e.g.\ \cite[Proposition~3.5.2]{PS:16}, which can often only be guaranteed upon introducing a shift.
Therefore, for $\mu \ge 0$, we consider the linearization
\begin{equation}\label{eq:linearization global}
\left\{
    \begin{aligned}
        \partial_t {v} +(\mu-\Delta) {v} +\nabla {p} =g_1, \enspace \mdiv {v} = 0, \tand {\partial_t {d}}  + (\mu-\Delta) {d} &= g_2, &&\tin (0,\infty) \times \cF_0,\\
        \mS({\ell})' + \mu \ell + \int_{\del \cS_0} \rT(v,p) N \rd \Gamma &= g_3, &&\tin (0,\infty),\\ 
        J_0({\omega})' + \mu \omega + \int_{\del \cS_0} y \times \rT(v,p) N \rd \Gamma &= g_4, &&\tin (0,\infty),\\
        {v}&=\ell + \omega \times y, &&\ton (0,\infty) \times \partial\cS_0,\\
        {v}&=0, &&\ton (0,\infty) \times \partial \cO,\\
        \partial_\nu d &= 0, &&\ton (0,\infty) \times \partial \cF_0,\\
        v(0) = v_0, \enspace d(0) = d_0 - d_*, \enspace \ell(0) = \ell_0, \tand \omega(0) &= \omega_0.
    \end{aligned}
\right.
\end{equation}

Maity and Tucsnak studied in \cite{MT:18} the interaction problem of a rigid body immersed in a viscous incompressible fluid.
Their handling of the linearized problem relies on a so-called ``monolithic approach'', i.e., the fluid and rigid body equations are considered as one complete system, and an equivalent formulation in terms of the {\em fluid-structure operator} and the pressure is provided.
The investigation of the {\em fluid-structure semigroup} relies on a decoupling method.
For a decoupling based on viewing linearized fluid-structure problems as boundary controlled fluid systems with dynamic boundary feedback, we also refer to \cite{MT:17}.
Another way prove the following result would be to use the decoupling approach from \cite{BBH:23} based on a different lifting method and the observation that maximal regularity is preserved under similarity transforms.

Combining the considerations in \cite[Sections~3--5]{MT:18}, and employing the maximal $\rL^p([0,\infty))$-regularity of the Neumann Laplacian up to a shift together with the fact that the $d$-equation is decoupled in \eqref{eq:linearization global}, we get the following result.

\begin{prop}\label{prop:max zero infty reg up to shift}
Let $p,q \in (1,\infty)$ such that \eqref{eq:cond p and q} holds true, and consider $(g_1,g_2,g_3,g_4) \in \F_\infty$, $d_* \in \R^3$ constant with $|d_*| = 1$ as well as $(v_0,d_0,\ell_0,\omega_0) \in \rX_\gamma$.
Then there exists $\mu_0 \ge 0$ such that for all $\mu > \mu_0$, there exists a unique solution $(v,d,\ell,\omega,p) \in \tE_\infty$ to \eqref{eq:linearization global}.
Moreover, for $C > 0$, we get the estimate
\begin{equation*}
    \left\| (v,d,\ell,\omega,p) \right\|_{\tE_\infty} \le C \left(\| (g_1,g_2,g_3,g_4) \|_{\F_\infty} + \| (v_0,d_0 - d_*,\ell_0,\omega_0) \|_{\rX_\gamma}\right).
\end{equation*}
\end{prop}

\section{Proof of the local strong well-posedness}\label{sec:proof local well-posedness}

In this section, we prove the local strong well-posedness of the interaction problem with the rigid body.
We describe the underlying fixed point argument in more details in \autoref{ssec:fixed point arg}.
In \autoref{ssec:estimates nonlinear terms}, we first discuss the procedure to determine the diffeomorphisms $X$ and $Y$ and then show estimates of the related terms in order to establish estimates of the nonlinear terms in a second step.
Finally, \autoref{ssec:proof of local wp} is dedicated to concluding \autoref{thm:local wp reformulated}, and to deducing \autoref{thm:local wp} therefrom.

\subsection{The fixed point argument}\label{ssec:fixed point arg}
\

As already done previously, we denote the principle variable by $z = (v,d,\ell,\omega)$ and also employ the notation $\tz$, $z^*$ or $z_i$.
In the sequel, we will reformulate the question of finding a unique strong solution $(z,p)$ to the transformed system of equations \eqref{eq:cv1}--\eqref{eq:cv3} as a fixed point problem.
To this end, we recall the reference solution $(z^*,p^*)$ from \autoref{prop:reference sol}.
For a solution $(z,p)$ to \eqref{eq:cv1}--\eqref{eq:cv3}, we set
\begin{equation*}
    \hv \coloneqq v-v^{*}, \enspace \hd \coloneqq d-d^{*}, \enspace \hell \coloneqq \ell-\ell^{*}, \enspace \homega \coloneqq \omega-\omega^{*}, \tand \hp:= p-p^{*}.
\end{equation*}
Hence, $(\hz,\hp) = (\hv,\hd,\hell,\homega,\hp)$ is a solution to the system of equations
\begin{equation}\label{eq:syst for fixed point}
\left\{
    \begin{aligned}
        \partial_t {\hv} -\Delta {\hv} +\nabla {\widehat{p}} + B(d^{*}){\hd}&=F_1, \enspace \mdiv {\hv}=0, \tand {\partial_t {\hd}}  - \Delta {\hd} = F_2, &&\tin (0,T) \times \cF_0,\\
        \mS({\hell})' + \int_{\del \cS_0} \rT(\hv,\hp) N \rd \Gamma &= F_3, &&\tin (0,T),\\ 
        J_0({\homega})' + \int_{\del \cS_0} y \times \rT(\hv,\hp) N \rd \Gamma &=F_4, &&\tin (0,T),\\
        {\hv}&={\hell}+{\homega}\times y, &&\ton (0,T) \times \partial\cS_0,\\
        {\hv}&=0, &&\ton (0,T) \times \partial \cO,\\
        \partial_\nu \hd &= 0, &&\ton (0,T) \times \partial \cF_0,\\
        \hv(0) = 0, \enspace \hd(0) &= 0, \enspace \hell(0) = 0, \tand \homega(0) = 0.
    \end{aligned}
\right.
\end{equation}
Furthermore, for $B(d^*) \hd$ as in \eqref{eq:def B}, $\cB$ from \autoref{sec:change of var} and $\rT(v,p)$ as in \eqref{eq:lin stress tensor}, we define
\begin{equation*}
    \cQ(d^{*},\hd) \coloneqq B(d^{*})\hd-\cB(d^{*}+\hd)(d^{*}+\hd), \tand \cT(v,p) \coloneqq Q(t)^\top \rT(Q(t) v(t,y),p(t,y)) Q(t).
\end{equation*}
Additionally recalling the transformed terms from \autoref{sec:change of var}, we find that $F_1$, $F_2$, $F_3$ and $F_4$ are given by
\begin{equation}\label{eq:nonlinear terms F_i}
    \begin{aligned}
        F_1(\hz,\hp) &\coloneqq (\cL_1-\Delta)(\hv+v^{*}) -\cM(\hv+v^{*})-\cN(\hv+v^{*})- (\cG -\nabla)(\hp+p^{*}) + \cQ(d^{*},\hd),\\
        F_2(\hz) &\coloneqq (\cL_2-\Delta)(\hd+d^{*}) - (\nabla (\hd+d^{*})\cdot \partial_t Y + (v\cdot \nabla) (\hd+d^{*})) + |\nabla (\hd+d^{*})|^2 (\hd+d^{*}+d_*),\\
        F_3(\hz,\hp) &\coloneqq -\mS((\homega+\omega^{*})\times(\hell+\ell^{*}))\\
        &\quad +\int\limits_{\partial \cS_0} \left((\rT - \cT)(\hv,\hp) + Q(t)^\top \nabla \left(\hd+d^{*}\right) \left(\nabla\left( \hd+d^{*}\right)\right)^\top Q(t) \right) N\rd \Gamma, \tand\\
        F_4(\hz,\hp) &\coloneqq (\homega+\omega^{*})\times (J_0(\homega+\omega^{*}))\\
        &\quad + \int\limits_{\partial \cS_0} y \times \left((\rT - \cT)(\hv,\hp) + Q(t)^\top \nabla \left(\hd+d^{*}\right) \left(\nabla \left(\hd+d^{*}\right)\right)^\top Q(t)\right) N\rd \Gamma.
    \end{aligned}
\end{equation}

In the following, we fix $T_0$ and $R_0 > 0$, and we consider $0 < T \le T_0$ as well as $0 < R \le R_0$.
For $p, q \in (1,\infty)$ such that \eqref{eq:cond p and q} holds, and recalling $\E_T$ as well as $\E_T^p$ and $\tE_T$ from \eqref{eq:sol space} and \eqref{eq:adjusted max reg space}, we define
\begin{equation}\label{eq:space fixed point arg and solution map}
    \cK_T^R \coloneqq \left\{(\tz,\tp) \in \prescript{}{0}{\E_T} \times \E_T^{p} : \| (\tz,\tp) \|_{\tE_T} \le R\right\}, \tand \Phi_T^R \colon \cK_T^R \to \prescript{}{0}{\E_T} \times \E_T^{p}, \twith \Phi_T^R(\tz,\tp) \coloneqq (\hz,\hp),
\end{equation}
where $(\hz,\hp)$ denotes the solution to \eqref{eq:syst for fixed point} with right-hand sides $F_1(\tz,\tp)$, $F_2(\tz)$, $F_3(\tz,\tp)$ and $F_4(\tz,\tp)$, and $(\tz,\tp) \in \cK_T^R$. 
It follows from \autoref{prop:max reg hom IVs} that $\Phi_T^R$ is well-defined provided $(F_1,F_2,F_3,F_4) \in \F_T$, where $\F_T$ is defined in \eqref{eq:data space}.
With regard to \eqref{eq:syst for fixed point}, the existence and uniqueness of a solution to \eqref{eq:cv1}--\eqref{eq:cv3} is equivalent to $\Phi_T^R$ admitting a unique fixed point upon adding the reference solution $(z^*,p^*)$.

\subsection{Estimates of the nonlinear terms}\label{ssec:estimates nonlinear terms}
\ 

Throughout this section, for $T_0 > 0$ and $R_0 > 0$ fixed, let $T \in (0,T_0]$ and $R \in (0,R_0]$.
We consider $(\tz,\tp)$, $(\tz_1,\tp_1)$, $(\tz_2,\tp_2) \in \cK_T^R$, and for the reference solution $(z^*,p^*)$ resulting from \autoref{prop:reference sol}, we define $(z,p) \coloneqq (\tv + v^*,\td + d^*,\tell + \ell^*,\tomega + \omega^*,\tp + p^*)$ and set $(z_i,p_i)$, $i=1,2$, likewise.
It then follows that
\begin{equation}\label{eq:est sol}
    \| (z,p) \|_{\tE_T} \le R + C_T^*,
\end{equation}
where $C_T^*$ from \autoref{rem:norm of reference sol} is the norm of $(z^*,p^*)$.
We also denote by $C_0 \coloneqq \| z_0 \|_{\rX_\gamma}$ the norm of the initial values.
Next, we comment on the procedure to deduce $X$ and $Y$ from \autoref{sec:change of var}.

\begin{rem}\label{rem:procedure to determine diffeos}
Let $\ell$, $\omega \in \rW^{1,p}(0,T)^3$ be given.
\begin{enumerate}[(i)]
    \item First, recover $Q \in \rW^{2,p}(0,T)^{3 \times 3}$ from solving the initial value problem
    \begin{equation}\label{eq:deduction Q}
        \Dot{Q}^{\top}(t) = M(t) Q^{\top}(t), \enspace Q^{\top}(0) = \Id,
    \end{equation}
    where the matrix-valued function $M(t)$ fulfills $M(t) x = \omega(t) \times x$ for all $t \in (0,T)$ and $x \in \R^3$.
    \item Equipped with $Q$, we can deduce the original body velocities from
    \begin{equation}\label{eq:deduction h' and Omega}
        h'(t) = Q^{\top}(t) \ell(t) \tand \Omega(t) = Q^{\top}(t) \omega(t).
    \end{equation}
    As we assume that the center of gravity $h$ is located at the origin at time zero, i.e., $h(0) = 0$, it can be derived from $h(t) = \int_0^t h'(s) \rd s$.
    \item It is then possible to insert $b$ as defined in \eqref{eq:rhs b} into \eqref{eq:IVP X}, so we obtain $X$.
    As in \eqref{eq:rhs bY}, we set $b^{(Y)} \coloneqq \rJ_X^{-1} b(\cdot,X)$ and then solve \eqref{eq:IVP Y} to get the inverse diffeomorphism $Y$ of $X$.
\end{enumerate} 
\end{rem}

Concerning the notation, given $(\ell_1,\omega_1)$, $(\ell_2,\omega_2) \in \rW^{1,p}(0,T)^6$, the subscript $i \in \{1,2\}$ indicates that the respective objects correspond to $(\ell_i,\omega_i)$ and are deduced therefrom by the procedure described in \autoref{rem:procedure to determine diffeos}.
This is especially valid for the diffeomorphisms $X_i$ and $Y_i$ related to $\ell_i$ and $\omega_i$.
Moreover, we use $\| \cdot \|_\infty$ and $\| \cdot \|_{\infty,\infty}$ to denote the $\rL^\infty$-norm in time and in time and space, respectively, and we will also write $\| \cdot \|_{p,q}$ instead of $\| \cdot \|_{\rL^p(0,T;\rL^q(\cF_0))}$.
For a proof of the following lemma, we refer to \cite[Section~6.1]{GGH:13}.

\begin{lem}\label{lem:props of the transform}
Let $(\ell_1,\omega_1)$, $(\ell_2,\omega_2) \in \rW^{1,p}(0,T)^6$.
\begin{enumerate}[(a)]
    \item For $i=1,2$, we have $X_i$, $Y_i \in \rC^1\left(0,T;\rC^\infty(\R^3)^3\right)$ as well as the estimates
    \begin{equation*}
        \begin{aligned}
            \| \partial^{\alpha} X_i \|_{\infty,\infty} + \| \partial^{\alpha} Y_i \|_{\infty,\infty} &\le C, \tand\\
            \| \partial^{\beta}(X_1 - X_2) \|_{\infty,\infty} + \| \partial^{\beta}(Y_1 - Y_2) \|_{\infty,\infty} &\le CT (\| \ell_1 - \ell_2 \|_{\infty} + \| \omega_1 - \omega_2 \|_{\infty})
        \end{aligned}
    \end{equation*}
    for all multi-indices $\alpha$ and $\beta$ with $1 \le |\alpha| \le 3$ and $0 \le |\beta| \le 3$.
    The constants only depend on $K_i \coloneqq \| \ell_i \|_{\infty} + \| \omega_i \|_{\infty}$ and not directly on $\ell_i$ or $\omega_i$.
    In particular, for $i,k,m \in \{1,2,3\}$, we have
    \begin{equation*}
        \| \partial_k \partial_i Y_m \|_{\infty,\infty} \le CT(R + C_T^*).
    \end{equation*}
    \item For $(h'_i,\Omega_i)$ deduced from $(\ell_i,\omega_i)$ as described in \eqref{eq:deduction h' and Omega}, the estimates
    \begin{equation*}
        \| h'_1 - h'_2 \|_{\infty} \le C (\| \ell_1 - \ell_2 \|_\infty + \| \omega_1 - \omega_2 \|_\infty), \tand \| \Omega_1 - \Omega_2 \|_\infty \le C \| \omega_1 - \omega_2 \|_\infty
    \end{equation*}
    are valid.
    In particular, for $K_i$ as above the matrix $Q_i$ from \eqref{eq:deduction Q} satisfies the estimates
    \begin{equation*}
        \| Q_1 - Q_2 \|_\infty \le C T \| M_1 - M_2 \|_\infty \le C T \| \omega_1 - \omega_2 \|_\infty, \tand \| Q_i \|_\infty + \| Q_i^\top \|_\infty \le C(1 + T K_i \mre^{T K_i}).
    \end{equation*}
    \item For all multi-indices $\beta$ with $0 \le |\beta| \le 3$, $b_i$ defined in \eqref{eq:rhs b} and associated to $h'_i$ and $\Omega_i$ fulfills
    \begin{equation*}
        \| \partial^\beta b_i \|_{\infty,\infty} \le C, \tand \| \partial^\beta (b_1 - b_2) \|_{\infty,\infty} \le C(\| h'_1 - h'_2 \|_\infty + \| \Omega_1 - \Omega_2 \|_\infty).
    \end{equation*}
    \item For all multi-indices $\beta$ with $0 \le |\beta| \le 3$, $b^{(Y_i)}$ defined in \eqref{eq:rhs bY} and associated to $h'_i$ and $\Omega_i$ satisfies
    \begin{equation*}
        \left\| \partial^\beta (b^{(Y_1)} - b^{(Y_2)}) \right\|_{\rL^\infty(0,T;\rC(\R^3))} \le C \| \partial^\beta(b_1 - b_2) \|_{\rL^\infty(0,T;\rC^1(\R^3))}.
    \end{equation*}
\end{enumerate}
\end{lem}

Next, we provide estimates of the covariant and contravariant metric tensors $g^{ij}$ and $g_{ij}$ and of the Christoffel symbol $\Gamma_{jk}^i$ defined in \autoref{sec:change of var}.
For $\ell$, $\ell_1$, $\ell_2$ and $\omega$, $\omega_1$, $\omega_2$ resulting from $\tell$, $\tell_1$, $\tell_2$, $\ell^*$ and $\tomega$, $\tomega_1$, $\tomega_2$, $\omega^*$ as described at the beginning of the section, we recover the diffeomorphisms $X$, $X_1$, $X_2$ and $Y$, $Y_1$, $Y_2$ as made precise in \autoref{rem:procedure to determine diffeos}.
The proof of the following lemma can be found in \cite[p.~1417]{GGH:13}.

\begin{lem}\label{lem:ests covariant contrvariant christoffel}
For all multi-indices $0 \le |\alpha| \le 1$, the covariant and contravariant tensor as well as the Christoffel symbol associated to $X$, $X_1$, $X_2$ and $Y$, $Y_1$, $Y_2$ satisfy the estimates
\begin{equation*}
    \| \partial^\alpha g^{ij} \|_{\infty,\infty} + \| \partial^\alpha g_{ij} \|_{\infty,\infty} + \| \partial^\alpha \Gamma_{jk}^i \|_{\infty,\infty} \le C, \tand
\end{equation*}
\begin{equation*}
    \| \partial^\alpha ((g_1)^{ij} - (g_2)^{ij}) \|_{\infty,\infty} + \| \partial^\alpha ((g_1)_{ij} - (g_2)_{ij}) \|_{\infty,\infty} + \| \partial^\alpha ((\Gamma_1)_{jk}^i - (\Gamma_2)_{jk}^i) \|_{\infty,\infty} \le C T \| (\tell_1 - \tell_2,\tomega_1 - \tomega_2) \|_\infty.
\end{equation*}   
\end{lem}

In the following lemma, we discuss further estimates in order to handle the nonlinear terms.

\begin{lem}\label{lem:further aux ests}
Let $p,q \in (1,\infty)$ satisfy \eqref{eq:cond p and q}, and for the reference solution $(z^*,p^*)$ from \autoref{prop:reference sol} and $(\tz,\tp)$, $(\tz_i,\tp_i) \in \cK_T^R$, let $(z,p) = (\tz + z^*,\tp + p^*)$ and $(z_i,p_i) = (\tz_i + z^*,\tp_i + p^*)$, $i=1,2$.
Then
\begin{equation*}
    \| g^{jk} - \delta_{jk} \|_{\infty,\infty} \le C T(R + C_T^*), \tand \| \partial_j Y_k - \delta_{jk} \|_{\infty,\infty} \le C T (R + C_T^*).
\end{equation*}
\end{lem}

\begin{proof}
The identity transform $X(t,y) = y$ for all $t>0$ and $y \in \R^3$ corresponds to the situation of the body at rest, i.e., $\ell=\omega = 0$.
Therefore, \autoref{lem:ests covariant contrvariant christoffel}, $\rW^{1,p}(0,T) \hookrightarrow \rL^\infty(0,T)$ and \eqref{eq:est sol} yield
\begin{equation*}
    \| g^{jk} - \delta_{jk} \|_{\infty,\infty} \le CT \| (\ell,\omega) \|_{\infty} \le CT \| (\ell,\omega) \|_{\rW^{1,p}(0,T)} \le CT (R+C_T^*),
\end{equation*}
and the estimate of $\| \partial_j Y_k - \delta_{jk} \|_{\infty,\infty}$ follows likewise.
\end{proof}

The embeddings in the lemma below will be used frequently in the estimates of the nonlinear terms.

\begin{lem}\label{lem:embedding of max reg space}
Let $p, q \in (1,\infty)$ satisfy \eqref{eq:cond p and q}, and set $\rY_0 \coloneqq \rL^q_\sigma(\cF_0) \times \rL^q(\cF_0)^3 \times \R^3 \times \R^3$ as well as $\rY_1 \coloneqq \rW^{2,q}(\cF_0)^3 \cap \rL^q_\sigma(\cF_0) \times \rW^{2,q}(\cF_0)^3 \times \R^3 \times \R^3$.
\begin{enumerate}[(a)]
    \item For all $\beta \in (0,1)$, denoting by $\rY_\beta \coloneqq (\rY_0,\rY_1)_\beta$ the complex interpolation space, the embedding
    \begin{equation*}
        \E_T \hookrightarrow \rW^{1,p}(0,T;\rY_0) \cap \rL^p(0,T;\rY_1) \hookrightarrow \rH^{\beta,p}(0,T;\rY_{1-\beta}) \hookrightarrow \rH^{\beta,p}\left(0,T;\rH^{2(1-\beta),q}(\cF_0)^6 \times \R^6\right)
    \end{equation*}
    is valid.
    In particular, we have $\E_T^{d} \hookrightarrow \rL^\infty\left(0,T;\rW^{1,2q}(\cF_0)^3\right) \hookrightarrow \rL^\infty\left(0,T;\rW^{1,q}(\cF_0)^3\right)$.
    \item For $\rY_\gamma = (\rY_0,\rY_1)_{1-\nicefrac{1}{p},p}$, we have
    \begin{equation*}
        \E_T \hookrightarrow \rW^{1,p}(0,T;\rY_0) \cap \rL^p(0,T;\rY_1) \hookrightarrow \BUC([0,T];\rY_\gamma) \hookrightarrow \BUC\left([0,T];\rB_{qp}^{2-\nicefrac{2}{p}}(\cF_0)^6 \times \R^6\right).
    \end{equation*}
    In particular, $\E_T^d \hookrightarrow \BUC\left([0,T];\rC^1(\overline{\cF_0})^3\right) \hookrightarrow \rL^\infty\left(0,T;\rW^{1,\infty}(\cF_0)^3\right) \hookrightarrow \rL^\infty\left(0,T;\rL^\infty(\cF_0)^3\right)$.
\end{enumerate}
In general, only the embedding constants of the embeddings into $\rH^{\beta,p}(0,T;\rY_\beta)$ and $\BUC([0,T];\rY_\gamma)$ depend on time.
They can be chosen time-independent if zero initial values are considered.
\end{lem}

\begin{proof}
The embeddings in (a) follow from the mixed derivative theorem, see e.g.\ \cite[Corollary~4.5.10]{PS:16}, while the second embedding in (b) is a result of \cite[Theorem~III.4.10.2]{Ama:95}.
Concerning the embeddings of $\E_T^d$ in (a), there is $\beta \in (0,1)$ with $\nicefrac{1}{p} < \beta < \nicefrac{1}{2} - \nicefrac{3}{2q}$ in view of \eqref{eq:cond p and q}, yielding the assertion by Sobolev embeddings.
In (b), we use \eqref{eq:cond p and q} and \cite[Theorem~4.6.1]{Tri:78} to deduce $\rB_{qp}^{2-\nicefrac{2}{p}}(\cF_0) \hookrightarrow \rC^1(\overline{\cF_0})$.
\end{proof}

We start with the estimates of the nonlinear term $F_2$ as defined in \eqref{eq:nonlinear terms F_i}.

\begin{prop}\label{prop:nonlinear ests F1}
Let $p,q \in (1,\infty)$ satisfy \eqref{eq:cond p and q}, and for the reference solution $(z^*,p^*)$ from \autoref{prop:reference sol} and $(\tz,\tp)$, $(\tz_i,\tp_i) \in \cK_T^R$, let $(z,p) = (\tz + z^*,\tp + p^*)$ and $(z_i,p_i) = (\tz_i + z^*,\tp_i + p^*)$, $i=1,2$.
Recall $\Cmr > 0$ from \autoref{prop:max reg hom IVs}.
Then there are $C_2(R,T) > 0$ with  $C_2(R,T) < \nicefrac{R}{8 \Cmr}$ for $T>0$ sufficiently small and $L_2(R,T) > 0$ with $L_2(R,T) \to 0$ as $R \to 0$ and $T \to 0$ such that
\begin{equation*}
    \| F_2(\tz) \|_{\F_T^d} \le C_2(R,T), \tand \| F_2(\tz_1) - F_2(\tz_2) \|_{\F_T^d} \le L_2(R,T) \| \tz_1 - \tz_2 \|_{\E_T}.
\end{equation*}
\end{prop}

\begin{proof}
Throughout this proof, we heavily rely on $(\tz,\tp) \in \cK_T^R$, implying that the embedding constants are time-independent in view of \autoref{lem:embedding of max reg space}.
Hence, we also often split $v$ and $d$ into their respective parts with homogeneous initial values $\tv$ and $\td$ as well as the reference solution parts $v^*$ and $d^*$.
Let us also invoke some useful estimate at this stage.
For $\td \in \prescript{}{0}{\E_T^d}$, we deduce from H\"older's inequality as well as \autoref{lem:embedding of max reg space}(a) the existence of a $T$-independent constant $C > 0$ such that
\begin{equation}\label{eq:est of td in Lp W1q}
    \| \td \|_{\rL^p(0,T;\rW^{1,q}(\cF_0))} \le T^{\nicefrac{1}{p}} \| \td \|_{\rL^\infty(0,T;\rW^{1,q}(\cF_0))} \le C T^{\nicefrac{1}{p}} \| \td \|_{\E_T^d}.
\end{equation}

First, we employ \autoref{lem:props of the transform}(a) to estimate $\| \Delta Y_k \|_{\infty,\infty}$, \autoref{lem:further aux ests} to handle $\| g^{jk} - \delta_{jk} \|_{\infty,\infty}$ as well as \eqref{eq:est sol} in order to bound $\| d \|_{\E_T^d}$, so we get
\begin{equation}\label{eq:est of cL_2}
    \begin{aligned}
        \| (\cL_2 - \Delta)d \|_{\F_T^d}
        &\le C \sup_{i,j,k}(\| g^{jk} - \delta_{jk} \|_{\infty,\infty} \| \partial_k \partial_j d_i \|_{p,q} + \| \Delta Y_k \|_{\infty,\infty} \| \partial_k d_i \|_{p,q})\\
        &\le C T (R + C_T^*) \| d \|_{\E_T^d}\\
        &\le C T (R + C_T^*)^2.
    \end{aligned}
\end{equation}

\autoref{lem:embedding of max reg space}(a) allows us to estimate $\| \tv \|_{\infty,\infty}$ by $\| \tz \|_{\E_T}$.
Thanks to \autoref{lem:est of BUC trace by MR space and IVs}, we obtain an estimate of $\| z^* \|_{\BUC([0,T];\rX_\gamma)}$ by the initial values and the reference solution.
In conjunction with \eqref{eq:est of td in Lp W1q}, we obtain
\begin{equation*}
    \begin{aligned}
        \| (v \cdot \nabla) d \|_{\F_T^d}
        &\le C(\| \tv \|_{\infty,\infty} + \| v^* \|_{\infty,\infty}) (\| \td \|_{\rL^p(0,T;\rW^{1,q}(\cF_0))} + \| d^* \|_{\rL^p(0,T;\rW^{1,q}(\cF_0))})\\
        &\le C(\| \tz \|_{\E_T} + \| z^* \|_{\BUC([0,T];\rX_\gamma)}) (T^{\nicefrac{1}{p}} \| \td \|_{\E_T^d} +  \| d^* \|_{\E_T^d})\\
        &\le C(R + C_0 + C_T^*) (T^{\nicefrac{1}{p}} R + C_T^*).
    \end{aligned}
\end{equation*}

By virtue of \autoref{lem:props of the transform}(a), which yields the boundedness of $\partial_t Y$, and \eqref{eq:est of td in Lp W1q}, it follows that
\begin{equation*}
    \| \nabla d \cdot \partial_t Y \|_{\F_T^d} \le C \| d \|_{\rL^p(0,T;\rW^{1,q}(\cF_0))} \le C (T^{\nicefrac{1}{p}} \| \td \|_{\E_T^{d}} + \| d^* \|_{\E_T^{d}}) \le C (T^{\nicefrac{1}{p}} R + C_T^*).
\end{equation*}

Exploiting \autoref{lem:embedding of max reg space}(b) for estimating $\| \td \|_{\rL^\infty(0,T;\rW^{1,\infty}(\cF_0))}$ by $\| \td \|_{\E_T^d}$ as well as $\| d^* \|_{\rL^\infty(0,T;\rW^{1,\infty}(\cF_0))}$ by $\| d^* \|_{\BUC([0,T];\rX_\gamma^{d})}$, and using \autoref{lem:est of BUC trace by MR space and IVs} with regard to $\| d^* \|_{\BUC([0,T];\rX_\gamma^{d})}$, we infer that
\begin{equation}\label{eq:est of the trilin d term}
    \begin{aligned}
        \| |\nabla d|^2 d \|_{\F_T^d}
        &\le C \| d \|_{\infty,\infty} \| \nabla d \|_{\infty,\infty} \| \nabla d \|_{p,q}\\
        &\le C(\| \td \|_{\E_T^{d}}^2 + \| d^* \|_{\BUC([0,T];\rX_\gamma^{d})}^2)(T^{\nicefrac{1}{p}} \| \td \|_{\rL^\infty(0,T;\rW^{1,q}(\cF_0))} + \| d^* \|_{\rL^p(0,T;\rW^{1,q}(\cF_0))})\\
        &\le C(R^2 + (C_T^*)^2 + C_0^2)(T^{\nicefrac{1}{p}}R + C_T^*),
    \end{aligned}
\end{equation}
and $\| |\nabla d|^2 d_* \|_{\F_T^d}$ can be estimated similarly, so the first part of the assertion follows for some $C_2(R,T) > 0$.

As in \eqref{eq:est of cL_2}, also making use of \autoref{lem:ests covariant contrvariant christoffel} and \autoref{lem:props of the transform}(a) to control the differences $(g_1)^{jk} - (g_2)^{jk}$ and $\Delta Y_{1,k} - \Delta Y_{2,k}$ resulting from the shape of $\cL_2$ as introduced in \autoref{sec:change of var}, we find the estimates
\begin{equation*}
    \| (\cL_2^{(1)} - \Delta)(\td_1 - \td_2) \|_{\F_T^d} \le C(T(R+C_T^*)) \| \td_1 - \td_2 \|_{\E_T^{d}}, \tand \| (\cL_2^{(1)} - \cL_2^{(2)})(\td_2 + d^*) \|_{\F_T^d} \le C T \| \tz_1 - \tz_2 \|_{\E_T}.
\end{equation*}
In total, we thus get
\begin{equation*}
    \| (\cL_2^{(1)} - \Delta)(\td_1 + d^*) - (\cL_2^{(2)} - \Delta)(\td_2 + d^*) \|_{\F_T^d} \le C(T(1 + R + C_T^*)) \| \tz_1 - \tz_2 \|_{\E_T}.
\end{equation*}
Similar arguments as in the estimate \eqref{eq:est of the trilin d term} and the addition and subtraction of suitable terms lead to
\begin{equation*}
    \| (\tv_1 + v^*) \cdot \nabla (\td_1 + d^*) - (\tv_2 + v^*) \cdot \nabla (\td_2 + d^*) \|_{\F_T^d} \le (T^{\nicefrac{1}{p}}(R+C_0) + C_T^*) \| \tz_1 - \tz_2 \|_{\E_T}.
\end{equation*}
Next, additionally exploiting \autoref{lem:props of the transform}(a) to estimate $\partial_t Y_1 - \partial_t Y_2$, and also relying on \eqref{eq:est of td in Lp W1q}, we conclude
\begin{equation*}
    \begin{aligned}
        &\quad \| \nabla (\td_1 + d^*) \partial_t Y_1 - \nabla (\td_2 + d^*) \partial_t Y_2 \|_{\F_T^d}\\
        &\le \| \nabla (\td_1 + d^*) (\partial_t Y_1 - \partial_t Y_2) \|_{p,q} + \| \nabla (\td_1 - \td_2) \partial_t Y_2 \|_{p,q}\\
        &\le C(T^{\nicefrac{1}{p}}R + C_T^*) T \| \tz_1 - \tz_2 \|_{\E_T} + C T^{\nicefrac{1}{p}} \| \td_1 - \td_2 \|_{\E_T^{d}}.
    \end{aligned}
\end{equation*}
Proceeding analogously as in \eqref{eq:est of the trilin d term} upon adding and subtracting terms, we get
\begin{equation*}
    \| |\nabla (\td_1 + d^*)|^2 (\td_1 + d^*) - |\nabla (\td_2 + d^*)|^2 (\td_2 + d^*) \|_{\F_T^d} \le C(T^{\nicefrac{1}{p}}(R^2 + (C_T^*)^2 + C_0^2) +C_0 C_T^*) \| \td_1 - \td_2 \|_{\E_T^{d}}.
\end{equation*}
Again, let us observe that the term involving $d_*$ can be dealt with likewise.
\end{proof}

For the treatment of $F_1$, it is crucial to provide estimates of the term $\cQ$.

\begin{lem}\label{lem:nonlinear ests Q}
Let $p,q \in (1,\infty)$ satisfy \eqref{eq:cond p and q}, and for the reference solution $(z^*,p^*)$ from \autoref{prop:reference sol} and $(\tz,\tp)$, $(\tz_i,\tp_i) \in \cK_T^R$, let $(z,p) = (\tz + z^*,\tp + p^*)$ and $(z_i,p_i) = (\tz_i + z^*,\tp_i + p^*)$, $i=1,2$.
Recall $\Cmr > 0$ from \autoref{prop:max reg hom IVs}.
Then there are $C_{\cQ}(R,T) > 0$ with  $C_{\cQ}(R,T) < \nicefrac{R}{16 \Cmr}$ for $T>0$ sufficiently small and $L_{\cQ}(R,T) > 0$ with $L_{\cQ}(R,T) \to 0$ as $R \to 0$ and $T \to 0$ such that
\begin{equation*}
    \| \cQ(\td,d^*) \|_{\F_T^v} \le C_{\cQ}(R,T), \tand \| \cQ^{(1)}(\td_1,d^*) - \cQ^{(2)}(\td_2,d^*) \|_{\F_T^v} \le L_{\cQ}(R,T) \| \tz_1 - \tz_2 \|_{\E_T}.
\end{equation*}
\end{lem}

\begin{proof}
First, we expand $\cQ(\td,d^*) = B(d^*) \td - \cB(d)d = B(d^*) \td - \cB(\td + d^*)(\td + d^*)$ as follows:
\begin{equation}\label{eq:shapes Q_I Q_II Q_III}
    \begin{aligned}
        \cQ(\td,d^*)
        &= B(d^*) \td - \cB(\td + d^*)(\td + d^*)\\
        &= (B(d^*) - B(\td + d^*))(\td + d^*) - B(d^*) d^* + (B(\td + d^*) - \cB(\td + d^*))(\td + d^*)\\
        &=: \cQ_\rI + \cQ_{\rII} + \cQ_{\rIII}.
    \end{aligned}
\end{equation}
Regarding $\cQ_\rI$, using the Einstein sum convention, we obtain
\begin{equation*}
    \cQ_\rI = (B(d^*) - B(\td + d^*))(\td + d^*) = - \partial_i \td_l\Delta(\td + d^*)_l - \partial_k \td_l \partial_k \partial_i (\td + d^*)_l.
\end{equation*}
An application of \autoref{lem:embedding of max reg space}(b) providing an estimate of $\| \partial_i \td_l \|_{\infty,\infty}$ by $\| \td \|_{\E_T^d}$ and the observation that both terms can be handled similarly lead to
\begin{equation}\label{eq:est of Q_I}
    \| \partial_i \td_l\Delta(\td + d^*)_l \|_{p,q} \le C \| \td \|_{\E_T^{d}} \| d \|_{\E_T^{d}} \le C R \| d \|_{\E_T^{d}}, \tand \| \cQ_\rI \|_{p,q} \le C R \| d \|_{\E_T^{d}}.
\end{equation}

\autoref{lem:embedding of max reg space}(b) enables us to estimate $\| \partial_i d_l^* \|_{\infty,\infty}$ by $\| d^* \|_{\BUC([0,T];\rX_\gamma^{d})}$, and in view of \autoref{lem:est of BUC trace by MR space and IVs}, we can estimate the resulting norm by norm of the initial values as well as the reference solution, so
\begin{equation}\label{eq:est of Q_II}
    \| \cQ_{\rII} \|_{p,q} \le C \| \partial_i d_l^* \|_{\infty,\infty} \| d^* \|_{\E_T^{d}} \le C C_T^* \| d^* \|_{\BUC([0,T];\rX_\gamma^{d})} \le C C_T^*(C_T^* + C_0).
\end{equation}

Concerning the estimate of $\cQ_{\rIII}$, we again often split $z$ into $\tz$ with homogeneous initial values and the reference solution $z^*$ as seen in the proof of \autoref{prop:nonlinear ests F1}.
Besides,  we further expand $\cQ_{\rIII}$ as follows:
\begin{equation}\label{eq:def I II III IV}
    \begin{aligned}
        \cQ_{\rIII}
        &= (B(\td + d^*) - \cB(\td + d^*))(\td + d^*)\\
        &= \bigl[\partial_i(\td + d^*)_l \Delta(\td + d^*)_l\bigr] + \bigl[\partial_k (\td + d^*)_l \partial_k \partial_i (\td + d^*)_l\bigr] - \bigl[(\cL_2(\td + d^*)_l \partial_m (\td + d^*)_l \partial_i Y_m)\bigr]\\
        &\quad - \bigl[(\partial_m (\td + d^*)_l \partial_k Y_m)(\partial_j \partial_m (\td + d^*)_l (\partial_k Y_j \partial_i Y_m) + 2 (\td + d^*)_l \partial_k \partial_i Y_m)\bigr]\\
        &=: \rI + \rII + \rIII + \rIV.
    \end{aligned}
\end{equation}
Recalling $d = \td + d^*$, we calculate
\begin{equation*}
    \rI + \rIII = \partial_i d_l \Delta d_l - \cL_2 d_l (\partial_m d_l \partial_i Y_m) = \partial_i d_l (\Delta - \cL_2) d_l - \cL_2 d_l \partial_m d_l (\partial_i Y_m - \delta_{im}).
\end{equation*}
The next step now is to estimate $\| \cL_2 d_l \|_{p,q}$.
In fact, making use of \autoref{lem:ests covariant contrvariant christoffel} to treat $\| g^{ij} \|_{\infty,\infty}$, \autoref{lem:props of the transform} to bound $\| \Delta Y_k \|_{\infty,\infty}$ as well as \eqref{eq:est sol} to estimate the resulting norm $\| d \|_{\E_T^d}$, we find that
\begin{equation}\label{eq:est of L2}
    \| \cL_2 d_l \|_{p,q} \le C\left(\| g^{ij} \|_{\infty,\infty} \| \partial_k \partial_j d \|_{p,q} + \| \Delta Y_k \|_{\infty,\infty} \| \partial_k d_i \|_{p,q}\right) \le C \| d \|_{\E_T^d} \le C (R + C_T^*).
\end{equation}
Furthermore, \autoref{lem:embedding of max reg space}(b) in the first step and \autoref{lem:est of BUC trace by MR space and IVs} in the second estimate result in
\begin{equation}\label{eq:est of del, dl in infty infty}
    \| \partial_m d_l \|_{\infty,\infty} \le C\left(\| d^* \|_{\BUC([0,T];\rX_\gamma^{d})} + \| \td \|_{\E_T^d}\right) \le C (C_T^* + C_0 + R).
\end{equation}
Invoking \autoref{lem:further aux ests} to estimate $\| \partial_i Y_m - \delta_{im} \|_{\infty,\infty}$, we derive from the two preceding estimates that
\begin{equation}\label{eq:est of L2 term}
    \| \cL_2 d_l \|_{p,q} \| \partial_m d_l \|_{\infty,\infty} \| \partial_i Y_m - \delta_{im} \|_{\infty,\infty} 
    \le C(C_T^* + C_0 + R) T(R + C_T^*)^2.
\end{equation}
Recalling \eqref{eq:est of del, dl in infty infty} and \eqref{eq:est of cL_2} for the first term, and plugging in the estimate from \eqref{eq:est of L2 term}, we infer that
\begin{equation}\label{eq:est of I+III}
    \begin{aligned}
        \| \rI + \rIII \|_{p,q}
        &\le C(\| \partial_i d_l \|_{\infty,\infty} \| (\Delta - \cL_2) d_l \|_{p,q} + \| \cL_2 d_l \|_{p,q} \| \partial_m d_l \|_{\infty,\infty} \| \partial_i Y_m - \delta_{im} \|_{\infty,\infty})\\
        &\le C(C_T^* + C_0 + R) T (R + C_T^*)^2.
    \end{aligned}
\end{equation}

We further split $\rIV$ into $\rIV = (\partial_m d_l \partial_k Y_m)(\partial_i \partial_m d_l (\partial_k Y_j \partial_i Y_m) + \partial_m d_l \partial_k \partial_i Y_m) =: \rV + \rVI$ and expand
\begin{equation*}
    \begin{aligned}
        \rV
        &= (\partial_m d_l \partial_k Y_m)(\partial_j \partial_m d_l)(\partial_k Y_j \partial_i Y_m)\\
        &= (\partial_m d_l \partial_k Y_m)(\partial_j \partial_m d_l \delta_{im} \delta_{jk} + \partial_j \partial_m d_l(\partial_j Y_k - \delta_{jk}) \partial_i Y_m + \partial_j \partial_m d_l(\partial_i Y_m - \delta_{im}) \delta_{jk})\\
        &=: \rV(1) + \rVII.
    \end{aligned}
\end{equation*}
Besides, we expand
\begin{equation*}
    \rV(1) = (\partial_m d_l (\partial_k Y_m - \delta_{km}) + \partial_m d_l \delta_{km}) \partial_j \partial_m d_l \delta_{im} \delta_{jk} \eqqcolon \rVIII + \rIX.
\end{equation*}
Observing that $\rIX = \partial_m d_l \delta_{km} \partial_j \partial_m d_l \delta_{im} \delta_{jk} = \partial_k d_l \partial_k \partial_i d_l$, and recalling the shape of $\rII$ from \eqref{eq:def I II III IV}, we find that $\rII$ and $\rIX$ cancel out.
As a result, it remains to treat the terms $\rVIII$, $\rVII$ and $\rVI$ to complete the estimates of $\cQ_{\rIII}$.
Similarly as in \eqref{eq:est of L2 term}, we first get
\begin{equation}\label{eq:est of V(1)_1}
    \| \rVIII \|_{p,q} \le C \| \partial_k d_l \|_{\infty,\infty} \| \partial_k Y_m - \delta_{km} \|_{\infty,\infty} \| \partial_j \partial_m d_l \|_{p,q} \le C(C_T^* + C_0 + R) T (R + C_T^*)^2.
\end{equation}
Likewise, also using \autoref{lem:props of the transform}(a) to bound $\| \partial_i Y_m \|_{\infty,\infty}$, we argue that
\begin{equation}\label{eq:est of V(2)}
    \begin{aligned}
        \| \rVII \|_{p,q}
        &\le C(\| \partial_j \partial_m d_l \|_{p,q} \| \partial_j Y_k - \delta_{jk} \|_{\infty,\infty} \| \partial_i Y_m \|_{\infty,\infty} + \| \partial_j \partial_m d_l \|_{p,q} \| \partial_j Y_k - \delta_{im} \|_{\infty,\infty})\\
        &\le C T(R + C_T^*)^2.
    \end{aligned}
\end{equation}
By virtue of \autoref{lem:props of the transform}(a), yielding the desired estimate of $\| \partial_k \partial_i Y_m \|_{\infty,\infty}$, we finally obtain
\begin{equation}\label{eq:est of VI}
    \| \rVI \|_{p,q} \le C \| \partial_k \partial_i Y_m \|_{\infty,\infty} \| \partial_m d_l \|_{p,q} \le C T(R + C_T^*)^2.
\end{equation}
In summary, recalling the shape of $\cQ_{\rIII}$ from \eqref{eq:def I II III IV} and concatenating \eqref{eq:est of I+III}, \eqref{eq:est of V(1)_1}, \eqref{eq:est of V(2)} as well as \eqref{eq:est of VI}, we find that $\cQ_{\rIII}$ satisfies an estimate of the desired form.
Together with \eqref{eq:est of Q_I} and \eqref{eq:est of Q_II}, this implies the existence of $C_{\cQ}(R,T) > 0$ such that the first part of the assertion follows.

For the Lipschitz estimate, we start by adding and subtracting terms in order to get a more advantageous representation.
More precisely, we write
\begin{equation*}
    \begin{aligned}
        \cQ^{(1)}(\td_1,d^*) - \cQ^{(2)}(\td_2,d^*) 
        &= [B(d^*) - B(\td_1 + d^*)](\td_1 - \td_2) + [B(\td_1 + d^*)\\
        &\quad - \cB^{(1)}(\td_1 + d^*)](\td_1 - \td_2) + [\cB^{(2)}(\td_2 + d^*) - \cB^{(1)}(\td_1 + d^*)](\td_2 + d^*)\\
        &=: \cR_{\rI} + \cR_{\rII} + \cR_{\rIII}.
    \end{aligned}
\end{equation*}
The terms $\cR_{\rI}$ and $\cR_{\rII}$ are similar to $\cQ_\rI$ and $\cQ_{\rIII}$ from \eqref{eq:shapes Q_I Q_II Q_III}, but the operator $B$ and its transformed version $\cB$ are applied to $\td_1 - \td_2$ which has zero initial values.
As in \eqref{eq:est of Q_I}, we obtain
\begin{equation*}
    \| \cR_{\rI} \|_{p,q} \le C R \| \td_1 - \td_2 \|_{\E_T^{d}}.
\end{equation*}
Concerning $\cR_{\rII}$, we use the same splitting in $\rI$, $\rII$, $\rIII$ and $\rIV$ as seen for $\cQ_{\rIII}$ in \eqref{eq:def I II III IV}, denoted by $\mathrm{i}$, $\mathrm{ii}$, $\mathrm{iii}$ and $\mathrm{iv}$ here for distinction.
Making use of \eqref{eq:est of cL_2} and \eqref{eq:est of del, dl in infty infty} for the first part, and recalling \eqref{eq:est of L2} as well as \eqref{eq:est of del, dl in infty infty}, and \autoref{lem:further aux ests} to estimate $\| \partial_i Y_m - \delta_{im} \|_{\infty,\infty}$, we derive that
\begin{equation*}
    \begin{aligned}
        &\quad \| \mathrm{i} + \mathrm{iii} \|_{p,q}\\
        &\le C(\| \partial_i (d_1)_l \|_{\infty,\infty} \| (\Delta - \cL_2^{(1)})(\td_1 - \td_2)_l \|_{p,q} + \| \cL_2^{(1)}(\td_1 - \td_2)_l \|_{p,q} \| \partial_m (d_1)_l \|_{\infty,\infty} \| (\partial_i (Y_1)_m - \delta_{im} \|_{\infty,\infty})\\
        &\le C(C_T^* + C_0 + R) T(R + C_T^*) \| \td_1 - \td_2 \|_{\E_T^{d}}.
    \end{aligned}
\end{equation*}
Also expanding $\mathrm{iv} = \mathrm{v} + \mathrm{vi}$, $\mathrm{v} = \mathrm{v}(1) + \mathrm{vii}$ as well as $\mathrm{v}(1) = \mathrm{viii} + \mathrm{ix}$ and using that $\mathrm{ii}$ and $\mathrm{ix}$ cancel out, it remains to estimate $\mathrm{viii}$, $\mathrm{vii}$ and $\mathrm{vi}$ given by
\begin{equation*}
    \begin{aligned}
        \mathrm{viii}
        &= \partial_m (d_1)_l (\partial_k (Y_1)_m - \delta_{km}) \partial_j (d_1)_m (\td_1 - \td_2)_l \delta_{im} \delta_{jk},\\
        \mathrm{vii}
        &= (\partial_m (d_1)_l \partial_k (Y_1)_m) (\partial_j \partial_m (\td_1 - \td_2)_l (\partial_j (Y_1)_k - \delta_{jk}) \partial_i (Y_1)_m\\
        &\quad + \partial_j \partial_m (\td_1 - \td_2)_l (\partial_i (Y_1)_m - \delta_{im}) \delta_{jk}), \enspace \text{and}\\
        \mathrm{vi}
        &= (\partial_m (d_1)_l \partial_k (Y_1)_m \partial_m (d_1)_l \partial_k \partial_i ((Y_1)_m - (Y_2)_m)).
    \end{aligned}
\end{equation*}
Arguing as in \eqref{eq:est of V(1)_1}, with $d$ replaced by $\td_1 - \td_2$, we get 
\begin{equation*}
    \| \mathrm{viii} \|_{p,q} \le C(C_T^* + C_0 + R)^2 T(R + C_T^*) \| \td_1 - \td_2 \|_{\E_T^d}.
\end{equation*}
Mimicking the procedure to obtain \eqref{eq:est of V(2)}, we infer that 
\begin{equation*}
    \| \mathrm{vii} \|_{p,q} \le C (C_T^* + C_0 + R + 1) T (R + C_T^*) \| \td_1 - \td_1 \|_{\E_T^d}.
\end{equation*}
Finally, using similar arguments as for \eqref{eq:est of VI}, and invoking \eqref{eq:est of del, dl in infty infty}, we conclude the estimate
\begin{equation*}
    \begin{aligned}
        \| \mathrm{vi} \|_{p,q}
        &\le C \| \partial_m d_1 \|_{\infty,\infty}^2 \| \partial_k \partial_i ((Y_1)_m - (Y_2)_m) \|_{p,q}\\
        &\le C(C_T^* + C_0 + R)^2 T \| (\tell_1 - \tell_2,\tomega_1 - \tomega_2) \|_{\infty}\\
        &\le C(C_T^* + C_0 + R)^2 T \| \tz_1 - \tz_2 \|_{\E_T}.
    \end{aligned}
\end{equation*}
A concatenation of the previous estimates yields an estimate of $\| \cR_{\rII} \|_{p,q}$ of the desired form.
With respect to $\cQ^{(1)}(\td_1,d^*) - \cQ^{(2)}(\td_2,d^*)$, it remains to estimate $\cR_{\rIII}$.
To this end, we first consider the difference in the first addend.
Recalling that $d_i = \td_i + d^*$ and using the Einstein sum convention, we get
\begin{equation*}
    \begin{aligned}
        &\quad \cL_2^{(2)} (d_2)_l \partial_m (d_2)_l \partial_i (Y_2)_m - \cL_2^{(1)} (d_2)_l \partial_m (d_1)_l \partial_i (Y_1)_m\\
        &= \cL_2^{(2)} (d_2)_l \partial_m (d_2)_l (\partial_i (Y_2)_m - \partial_i (Y_1)_m) + \cL_2^{(2)} (d_2)_l \partial_m (d_2 - d_1)_l \partial_i (Y_1)_m\\
        &\quad + (\cL_2^{(2)} - \cL_2^{(1)}) (d_2)_l \partial_m (d_1)_l \partial_i (Y_1)_m.
    \end{aligned}
\end{equation*}
Exploiting \eqref{eq:est of L2} for the estimate of $\| \cL_2^{(2)} (d_2)_l \|_{p,q}$ and \eqref{eq:est of del, dl in infty infty} to control $\| \partial_m (d_2)_l \|_{\infty,\infty}$, and estimating the difference $(\partial_i (Y_2)_m - \partial_i (Y_1)_m)$ by \autoref{lem:props of the transform}(a), we deduce that
\begin{equation*}
    \| \cL_2^{(2)} (d_2)_l \partial_m (d_2)_l (\partial_i (Y_2)_m - \partial_i (Y_1)_m) \|_{p,q} \le C (C_T^* + C_0 + R) T (R + C_T^*) \| \tz_1 - \tz_2 \|_{\E_T}.
\end{equation*}
Thanks to \autoref{lem:embedding of max reg space}(b), we further obtain
\begin{equation*}
    \| \cL_2^{(2)} (d_2)_l \partial_m (d_2 - d_1)_l \partial_i (Y_1)_m \|_{p,q} \le C(R + C_T^*) T^{\nicefrac{1}{p}} \| \td_1 - \td_2 \|_{\E_T^{d}}.
\end{equation*}
The last term can be estimated in a similar way.
\end{proof}

The result below on estimates of $F_1$ as introduced in \eqref{eq:nonlinear terms F_i} is a consequence of \autoref{lem:nonlinear ests Q} and the estimates of the other transformed terms from \autoref{sec:change of var} which can be obtained in a similar way as the estimates in \autoref{prop:nonlinear ests F1} or \autoref{lem:nonlinear ests Q}.
We also refer to \cite[Lemmas~6.6 and 6.8]{GGH:13}.

\begin{prop}\label{prop:nonlinear ests F0}
Let $p,q \in (1,\infty)$ satisfy \eqref{eq:cond p and q}, and for the reference solution $(z^*,p^*)$ from \autoref{prop:reference sol} and $(\tz,\tp)$, $(\tz_i,\tp_i) \in \cK_T^R$, let $(z,p) = (\tz + z^*,\tp + p^*)$ and $(z_i,p_i) = (\tz_i + z^*,\tp_i + p^*)$, $i=1,2$.
Recall $\Cmr > 0$ from \autoref{prop:max reg hom IVs}.
Then there are $C_1(R,T) > 0$ with  $C_1(R,T) < \nicefrac{R}{8 \Cmr}$ for $T>0$ sufficiently small and $L_1(R,T) > 0$ with $L_1(R,T) \to 0$ as $R \to 0$ and $T \to 0$ such that
\begin{equation*}
    \| F_1(\tz,\tp) \|_{\F_T^v} \le C_1(R,T), \tand \| F_1(\tz_1,\tp_1) - F_1(\tz_2,\tp_2) \|_{\F_T^v} \le L_1(R,T) \| (\tz_1,\tp_1) - (\tz_2,\tp_2) \|_{\tE_T}.
\end{equation*}
\end{prop}

Finally, we prove estimates of the terms $F_3$ and $F_4$ appearing in the equations of the rigid body and defined in \eqref{eq:nonlinear terms F_i}.

\begin{prop}\label{prop:nonlinear ests F2 and F3}
Let $p,q \in (1,\infty)$ satisfy \eqref{eq:cond p and q}, and for the reference solution $(z^*,p^*)$ from \autoref{prop:reference sol} and $(\tz,\tp)$, $(\tz_i,\tp_i) \in \cK_T^R$, let $(z,p) = (\tz + z^*,\tp + p^*)$ and $(z_i,p_i) = (\tz_i + z^*,\tp_i + p^*)$, $i=1,2$.
Recall $\Cmr > 0$ from \autoref{prop:max reg hom IVs}.
Then there are $C_3(R,T)$, $C_4(R,T) > 0$ with  $C_3(R,T)$, $C_4(R,T) < \nicefrac{R}{8 \Cmr}$ for $T>0$ sufficiently small and $L_3(R,T)$, $L_4(R,T) > 0$ with $L_3(R,T)$, $L_4(R,T) \to 0$ as $R \to 0$ and $T \to 0$ such that
\begin{equation*}
    \begin{aligned}
        \| F_3(\tz,\tp) \|_{\F_T^\ell} 
        &\le C_3(R,T), \tand \| F_3(\tz_1,\tp_1) - F_3(\tz_2,\tp_2) \|_{\F_T^\ell} \le L_3(R,T) \| (\tz_1,\tp_1) - (\tz_2,\tp_2) \|_{\tE_T}, \taswellas\\
        \| F_4(\tz,\tp) \|_{\F_T^\ell} 
        &\le C_4(R,T), \tand \| F_4(\tz_1,\tp_1) - F_4(\tz_2,\tp_2) \|_{\F_T^\ell} \le L_4(R,T) \| (\tz_1,\tp_1) - (\tz_2,\tp_2) \|_{\tE_T}.
    \end{aligned}
\end{equation*}
\end{prop}

\begin{proof}
The estimates of $\omega \times \ell$ and $\omega \times (J_0 \omega)$ readily follow, so it remains to investigate the surface integrals.
For this, we introduce $\cJ \colon \rW^{\eps + \nicefrac{1}{q},q}(\cF_0)^{3 \times 3} \to \R^6$ defined by 
\begin{equation}\label{eq:joint est surface ints}
    \cJ(h) := \begin{pmatrix}
        \int_{\partial \cS_0} h N \rd \Gamma\\
        \int_{\partial \cS_0} y \times h N \rd \Gamma
    \end{pmatrix}, \twith | \cJ(h) | \le C \| h \|_{\rW^{\eps + \nicefrac{1}{q},q}(\cF_0)} \tfor h \in \rW^{\eps + \nicefrac{1}{q},q}(\cF_0)^{3 \times 3}.
\end{equation}
With regard to the surface integrals in $F_3$ and $F_4$, we observe that they consist of two parts, namely the part associated to the difference in the linearized stress tensor, $\rT - \cT$, and the nonlinear term resulting from the $d$-part in the stress tensor.
Concerning the first part, we observe that the estimates have already been established in \cite[Lemmas~6.6, 6.7 and 6.8]{GGH:13} by invoking \eqref{eq:joint est surface ints}.
Therefore, we can restrict ourselves to the respective estimates for the remaining $d$-part in the surface integrals.
Here, we use a direct approach to estimate the terms instead of \eqref{eq:joint est surface ints}.
From the continuity of the trace, we conclude
\begin{equation}\label{eq:spatial est d part surface int}
    \begin{aligned}
        \left| \int_{\del\cS_0} Q^\top (\nabla(\td + d^*)) ((\nabla (\td + d^*)))^\top Q N \rd \Gamma\right| &\le C |Q|^2 \| \td + d^* \|_{\rH^{\nicefrac{3}{2}}(\cF_0)}^2.
    \end{aligned}
\end{equation}
It further follows from the mixed derivative theorem and the condition $\nicefrac{2}{p} + \nicefrac{3}{q} < 1$ that
\begin{equation}\label{eq:emb max reg space into Linfty H32}
    \E_T^{d} \hookrightarrow \rH^{\theta,p}(0,T;\rH^{2(1-\theta),q}(\cF_0)^3) \hookrightarrow \rL^\infty(0,T;\rH^{\nicefrac{3}{2}}(\cF_0)^3), \tand \rX_\gamma^{d} \hookrightarrow \rH^{\nicefrac{3}{2}}(\cF_0)^3
\end{equation}
thanks to \cite[Theorem~4.6.1]{Tri:78}.
Applying the above equation \eqref{eq:spatial est d part surface int}, \autoref{lem:props of the transform}(b) to bound $\| Q \|_\infty$, the preceding embedding \eqref{eq:emb max reg space into Linfty H32} and \autoref{lem:est of BUC trace by MR space and IVs} for the estimate of $ \| d^* \|_{\BUC([0,T];\rX_\gamma^{d})}$ by the norm of the initial data and the reference solution, we derive that
\begin{equation}\label{eq:est d part surface ints}
    \begin{aligned}
        &\quad\left \| \int_{\del\cS_0} Q^\top (\nabla(\td + d^*)) ((\nabla (\td + d^*)))^\top Q N \rd \Gamma \right\|_p\\
        &\le C \| \td + d^* \|_{\rL^\infty(0,T;\rH^{\nicefrac{3}{2}}(\cF_0))} \| \td + d^* \|_{\rL^p(0,T;\rH^{\nicefrac{3}{2}}(\cF_0))}\\
        &\le C\left(\| \td \|_{\E_T^{d}} + \| d^* \|_{\BUC([0,T];\rX_\gamma^{d})}\right) \left(T^{\nicefrac{1}{p}} \| \td \|_{\rL^\infty(0,T;\rH^{\nicefrac{3}{2}}(\cF_0))} + \| d^* \|_{\E_T^{d}}\right)\\
        &\le C(R + C_T^* + C_0) (T^{\nicefrac{1}{p}}R + C_T^*).
    \end{aligned}
\end{equation}
Let us remark that the other surface integral involving the third addend of the stress tensor can be dealt with analogously.
The first part of the assertion then follows.

With respect to the Lipschitz estimate, we can expand the terms suitably to get a difference in each component.
For the differences in $Q$, we can use a similar estimate as \eqref{eq:est d part surface ints} together with \autoref{lem:props of the transform}(b) in order to get an estimate by the difference.
Concerning differences in $d$, a slight modification of the estimate is necessary.
In fact, as a preparation, we consider $d_1$, $d_2 \in \E_T$, where $d_1$ additionally satisfies $d_1(0) = 0$.
Let us observe that this corresponds to the situation of the difference being in $d_1$, i.e., $d_1 = \td_1 - \td_2$.
The other term can be handled completely analogously.
By \autoref{lem:embedding of max reg space}(b), we obtain in particular that $d_1$, $d_2 \in \BUC([0,T];\rC^1(\overline{\cF_0})^3)$.
Denoting by $q' \in (1,\infty)$ the H\"older conjugate of $q$, and using H\"older's inequality as well as continuity of the trace in addition to the aforementioned embedding, we get
\begin{equation*}
    \begin{aligned}
        \left|\int_{\del\cS_0} \nabla d_1 (\nabla d_2)^\top N \rd \Gamma \right|
        &\le C \| \nabla d_1 \|_{\rL^{q'}(\del\cS_0)} \| \nabla d_2 \|_{\rL^q(\del\cS_0)}\le C \| d_1 \|_{\rC^1(\overline{\cF_0})} \| d_2 \|_{\rW^{1+\nicefrac{1}{q} + \eps,q}(\cF_0)}\\
        &\le C \| d_1 \|_{\rC^1(\overline{\cF_0})} \| d_2 \|_{\rW^{2,q}(\cF_0)}.
    \end{aligned}
\end{equation*}
Therefore, employing H\"older's inequality in time in conjunction with \autoref{lem:props of the transform}(b) to control $\| Q \|_\infty$ and \autoref{lem:embedding of max reg space}(b) for an embedding, we infer that
\begin{equation*}
    \left\| \int_{\del\cS_0} Q^\top \nabla d_1 (\nabla d_2)^\top Q N \right\|_p \le C \| d_1 \|_{\rL^\infty(0,T;\rC^1(\overline{\cF_0}))} \| d_2 \|_{\rL^p(0,T;\rW^{2,q}(\cF_0))} \le C \| d_1 \|_{\E_T^d} \| d_2 \|_{\E_T^d}.
\end{equation*}
The embedding constant resulting from \autoref{lem:embedding of max reg space}(b) is $T$-independent thanks to the homogeneous initial values.
This shows how to estimate the terms with differences in the $d$-terms and then completes the proof as the terms corresponding to the second surface integral can be treated by analogy.
\end{proof}

\subsection{Proof of \autoref{thm:local wp}}\label{ssec:proof of local wp}
\

We first show \autoref{thm:local wp reformulated} and then deduce \autoref{thm:local wp} by performing the backward change of variables.

\begin{proof}[Proof of \autoref{thm:local wp reformulated}]
For $(\tz,\tp) \in \cK_T^R$, in view of \eqref{eq:syst for fixed point}, it follows from \autoref{prop:max reg hom IVs}, \autoref{prop:nonlinear ests F1}, \autoref{prop:nonlinear ests F0} and \autoref{prop:nonlinear ests F2 and F3} that
\begin{equation*}
    \| \Phi_T^R(\tz,\tp) \|_{\tE_T} \le \Cmr \| (F_1(\tz,\tp), F_2(\tz),F_3(\tz,\tp),F_4(\tz,\tp)) \|_{\F_T} \le \Cmr C(R,T),
\end{equation*}
where given $R>0$, we have $C(R,T) < \nicefrac{R}{2 \Cmr}$ for $T > 0$ sufficiently small.
On the other hand, the same arguments imply that for $(\tz_1,\tp_1)$, $(\tz_2,\tp_2) \in \cK_T^R$, we get
\begin{equation*}
    \| \Phi_T^R(\tz_1,\tp_1) - \Phi_T^R(\tz_2,\tp_2) \|_{\tE_T} \le \Cmr L(R,T) \| (\tz_1,\tp_1) - (\tz_2,\tp_2) \|_{\tE_T}.
\end{equation*}
Hence, by the time-independence of $\Cmr > 0$, we choose $R>0$ and $T>0$ sufficiently small such that $\Cmr C(R,T) \le \nicefrac{R}{2}$ and $\Cmr L(R,T) \le \nicefrac{1}{2}$, so $\Phi_T^R$ is a self map and contraction.
The contraction mapping principle yields the existence of a unique fixed point $(\hz,\hp) \in \prescript{}{0}{\E_T} \times \E_T^p$.
Adding $(z^*,p^*)$ from \autoref{prop:reference sol}, we get the unique solution $(z,p) = (\hz + z^*,\hp + p^*) \in \tE_T$ to \eqref{eq:cv1}--\eqref{eq:cv3}.
Let us observe that $(z,p) \in \tE_T$ corresponds to $(v,d,\ell,\omega,p)$ being in the regularity class as asserted in \autoref{thm:local wp reformulated}.
\end{proof}

\begin{proof}[Proof of \autoref{thm:local wp}]
From $(\ell,\omega) \in \rW^{1,p}(0,T)^6$ emerging from \autoref{thm:local wp reformulated}, we derive $h'$, $\Omega$ as well as the diffeomorphism $X$ as described in \autoref{rem:procedure to determine diffeos}.
Performing the backward change of variables and coordinates as pointed out in \autoref{sec:change of var}, we conclude the assertion of \autoref{thm:local wp} from \autoref{thm:local wp reformulated} upon adding $d_*$ in the $d$-component.
The uniqueness of the solution is a consequence of the uniqueness of the diffeomorphism.
Moreover, the average zero condition of the pressure is preserved in view of the definition of the transformed pressure $p$ on the fixed domain, see \autoref{sec:change of var}.
\end{proof}

\section{Proof of the global strong well-posedness close to equilibria}\label{sec:proof global strong wp}

In this section, we prove the second main result of this paper on the global strong well-posedness of the interaction problem \eqref{eq:LC-fluid equationsiso}--\eqref{eq:initial} for initial data close to constant equilibria.
To this end, we establish a maximal $\rL^p([0,\infty))$-regularity type result in  \autoref{ssec:analysis of lin probl and max reg} by using a splitting into the mean value-zero part and the average in the $d$-component in order to get exponential decay.
\autoref{ssec:ests of the nonlinear terms} is dedicated to showing estimates of the nonlinear terms tailored to the adjusted setting.
These estimates lead to the proof of the second main result in \autoref{ssec:proof of global wp close to equilibria} by means of a fixed point argument.

\subsection{Analysis of the linearized system}\label{ssec:analysis of lin probl and max reg}
\

Let us recall from \autoref{prop:max zero infty reg up to shift} that the maximal regularity of \eqref{eq:linearization global} is only obtained up to a shift.
It can be shown that we obtain maximal $\rL^p([0,\infty))$-regularity of the linearized system provided for the associated resolvent problem, we can establish that it admits a unique solution for all $\lambda \in \C_- \cup \{0\}$.
The resolvent problem takes the shape
\begin{equation}\label{eq:resolvent problem}
\left\{
    \begin{aligned}
        \lambda v -\Delta {v} +\nabla {p} &=g_1, \enspace \mdiv {v} = 0, \tand \lambda d - \Delta d = g_2, &&\tin \cF_0,\\
        \lambda \mS {\ell} + \int_{\del \cS_0} \rT(v,p) N \rd \Gamma &= g_3, \tand \lambda J_0 {\omega} + \int_{\del \cS_0} y \times \rT(v,p) N \rd \Gamma = g_4, &&\tin \R^3,\\ 
        {v}&=\ell + \omega \times y, \ton \del \cS_0, \enspace v = 0, \ton \del \cO, \tand \partial_\nu d = 0, &&\ton \partial \cF_0.
    \end{aligned}
\right.
\end{equation}
When employing the operator theoretic formulation as in \cite[Section~3]{MT:18} together with the Neumann Laplacian operator for the director part, it readily follows that the resulting operator has a compact resolvent, so the spectrum only consists of eigenvalues, and it suffices to consider the $\rL^2$-case.
Hence, considering $(v,d,\ell,\omega,p)$ solving \eqref{eq:resolvent problem} with $(g_1,g_2,g_3,g_4) = (0,0,0,0)$, we test the resulting fluid equation by $v$, integrate it over $\cF_0$, and employ $\mdiv v = 0$ as well as the boundary conditions $v = 0$ on $\partial \cO$ and $v = \ell + \omega \times y$.
Moreover, observing that $-\Delta v + \nabla p = -\mdiv (2 \D v - p \Id)$, we deduce from integration by parts, from $a \cdot (b \times c) = b \cdot (c \times a)$ and from $\ell$ and $\omega$ being constant in space that
\begin{equation}\label{eq:testing the fluid equation}
    \begin{aligned}
        0 
        &= - \int_{\cF_0} \mdiv  (2 \D v - p \Id) \cdot v \rd y = \int_{\cF_0} |\nabla v|^2 \rd y - \int_{\partial \cS_0} (2 \D v - p \Id) \cdot v N \rd \Gamma\\
        &= \int_{\cF_0} |\nabla v|^2 \rd y - \ell \cdot \int_{\partial \cS_0} (2 \D v - p \Id) N \rd \Gamma - \omega \cdot \int_{\partial \cS_0} y \times (2 \D v - p \Id) N \rd \Gamma.
    \end{aligned}
\end{equation}
Therefore, testing the resolvent problem with right-hand sides equal to zero by $(v,d,\ell,\omega)$, integrating by vector parts, invoking \eqref{eq:testing the fluid equation} and additionally using that $\partial_\nu d = 0$ on $\partial \cF_0$, we obtain 
\begin{equation}\label{eq:testing the eigenvalue equation}
    0 = \lambda\left(\| v \|_{\rL^2(\cF_0)}^2 + \| d \|_{\rL^2(\cF_0)}^2 + |\ell|^2 + |\omega|^2\right) + \| \nabla v \|_{\rL^2(\cF_0)}^2 + \| \nabla d \|_{\rL^2(\cF_0)}^2.
\end{equation}
As $\| \nabla v \|_{\rL^2(\cF_0)}^2$ and $\| \nabla d \|_{\rL^2(\cF_0)}^2$ are real and non-negative, it follows that $\lambda \in \R$ and $\lambda \le 0$.
In the case $\lambda = 0$, we derive from \eqref{eq:testing the eigenvalue equation} that $0 = \| \nabla v \|_{\rL^2(\cF_0)}^2 + \| \nabla d \|_{\rL^2(\cF_0)}^2$, i.e., $v$ and $d$ are constant on $\cF_0$.
In view of $v = 0$ on $\partial \cO$, we conclude that $v = 0$.
As the velocity of the fluid and the rigid body coincide on their interface, we deduce therefrom that $\ell + \omega \times y = 0$ on $\partial \cS_0$.
Similarly as in \autoref{ssec:energy}, using the 3D analogue of \cite[Lemma~2.1]{RT:19}, this implies that $\ell = \omega = 0$.
For the pressure $p$ with average zero, it also follows that $p = 0$.
On the other hand, in order to remedy the fact that $d$ can be constant, it is natural to invoke the space of average zero functions $\rL_0^q(\cF_0)$.
It readily follows that $\rL_0^q(\cF_0)$ is invariant under the semigroup $\mre^{\Delta_{\rN}t}$ and that the restriction of $\Delta_{\rN}$ to $\rL_0^q(\cF_0)$ is invertible.
This yields that when invoking $\rL_0^q(\cF_0)^3$ for the $d$-equation, the resolvent problem admits a unique solution for all $\lambda \in \C_+ \cup \{0\}$.

In order to account for the average zero condition in the $d$-part, we introduce some further notation.
In fact, for $\rX_0$, $\rX_1$ and $\rX_\gamma$ as made precise in \eqref{eq:ground and regularity space} and \eqref{eq:trace space}, respectively, and for $i \in \{0,1,\gamma\}$, we define
\begin{equation*}
    \rX_i^m \coloneqq \left\{z = (v,d,\ell,\omega) \in \rX_i : d \in \rL_0^q(\cF_0)^3\right\}.
\end{equation*}
Accordingly, for $\F_T^v$, $\F_T^\ell$ and $\F_T^\omega$ as made precise in \eqref{eq:data space}, we set
\begin{equation}\label{eq:data and sol space average zero}
    \begin{aligned}
        \F_T^m 
        &\coloneqq \F_T^v \times \rL^p\left(0,T;\rL_0^q(\cF_0)^3\right) \times \F_T^\ell \times \F_T^\omega,\\
        \E_T^m 
        &\coloneqq \rW^{1,p}(0,T;\rX_0^m) \cap \rL^p(0,T;\rX_1^m), \tand \tE_T^m \coloneqq \E_T^m \times \E_T^p.
    \end{aligned}
\end{equation}

We summarize the preceding discussion in the following proposition.

\begin{prop}\label{prop:max reg on average zero space}
Let $p,q \in (1,\infty)$ satisfy \eqref{eq:cond p and q}, and consider $d_* \in \R^3$ constant such that $|d_*| = 1$.
Then for all $\mu \ge 0$, $(g_1,g_2,g_3,g_4) \in \F_\infty^m$ and $(v_0,d_0 - d_*,\ell_0,\omega_0) \in \rX_\gamma^m$, there exists a unique solution $(v,d,\ell,\omega,p) \in \tE_\infty^m$ to \eqref{eq:linearization global}.
Besides, for $C > 0$, we obtain the estimate
\begin{equation*}
    \| (v,d,\ell,\omega,p) \|_{\tE_\infty^m} \le C \left(\| (g_1,g_2,g_3,g_4) \|_{\F_\infty^m} + \| (v_0,d_0 - d_*,\ell_0,\omega_0) \|_{\rX_\gamma^m}\right).
\end{equation*}
\end{prop}

As the resolvent set is open, we even obtain the existence of $\eta_0$ such that the semigroup associated to the operator $A_0^m$ corresponding to the linearized problem is exponentially stable, i.e., for $z = (v,d,\ell,\omega) \in \rX_0^m$, we have
\begin{equation}\label{eq:exp stability restricted semigroup}
    \| \mre^{-A_0^m t} (v,d,\ell,\omega) \|_{\rX_0^m} \le C \mre^{-\eta_0 t} \| (v,d,\ell,\omega) \|_{\rX_0^m}.
\end{equation}

The average zero condition is generally not preserved by the equation of the director $d$ in the interaction problem, so we decompose the equation into two parts.
More precisely, for a function $f \in \rL^1(\cF_0)$, we introduce the splitting 
\begin{equation}\label{eq:splitting in mean zero and avg}
    f = f_m + f_\avg, \enspace \text{where} \enspace f_\avg \coloneqq \frac{1}{|\cF_0|} \int_{\cF_0} f(y) \rd y \enspace \text{and} \enspace f_m \coloneqq f - f_\avg,
\end{equation}
resulting in $f_m \in \rL_0^1(\cF_0)$.
For $\F_T$ and $\tE_T^m$ as defined in \eqref{eq:data space} and \eqref{eq:data and sol space average zero} as well as $z = (v,d,\ell,\omega)$, we set
\begin{equation*}
    \mre^{-\eta(\cdot)} \F_\infty \coloneqq \left\{z \in \F_\infty : \mre^{\eta(\cdot)} z \in \F_\infty\right\}, \tand \mre^{-\eta(\cdot)} \tE_\infty^m \coloneqq \left\{(z,p) \in \tE_\infty^m : \mre^{\eta(\cdot)} (z,p) \in \tE_\infty^m\right\}.
\end{equation*}
The maximal regularity type result concerning the linearized problem \eqref{eq:linearization global} then reads as follows.
Let us observe that we may consider \eqref{eq:linearization global} with $\mu = 0$ by virtue of \autoref{prop:max reg on average zero space}.

\begin{prop}\label{prop:max reg type result global}
Let $p,q \in (1,\infty)$ satisfy \eqref{eq:cond p and q}, let $d_* \in \R^3$ be constant such that $|d_*| = 1$, and consider $\eta \in (0,\eta_0)$, where $\eta_0$ is the constant from \eqref{eq:exp stability restricted semigroup}.
Then for all $(v_0,d_0-d_*,\ell_0,\omega_0) \in \rX_\gamma^m$ and for all $(g_1,g_2,g_3,g_4) \in \mre^{-\eta(\cdot)} \F_\infty$ with $g_{2,\avg} \in \rL^1(0,\infty)$, \eqref{eq:linearization global} with $\mu = 0$ has a unique solution $(v,d,\ell,\omega,p)$ such that
\begin{equation*}
    (v,d_m,\ell,\omega,p) \in \mre^{-\eta(\cdot)} \tE_\infty^m, \tand d_\avg \in \rL^\infty(0,\infty),
\end{equation*}
and there exists $\Cexp = \Cexp(p,q,\eta) > 0$ such that
\begin{equation*}
    \begin{aligned}
        & \quad \left\| \mre^{\eta(\cdot)}(v,d_m,\ell,\omega,p) \right\|_{\tE_\infty^m} + \left\| d_\avg \right\|_{\rL^\infty(0,\infty)} + \left\| \mre^{\eta(\cdot)} \partial_t d_\avg \right\|_{\rL^p(0,\infty)}\\
        &\le \Cexp\left(\left\| (v_0,d_0 - d_*,\ell_0,\omega_0) \right\|_{\rX_\gamma^m} + \left\| \mre^{\eta(\cdot)} (g_1,g_2,g_3,g_4) \right\|_{\F_\infty} + \left\| g_{2,\avg} \right\|_{\rL^1(0,\infty)}\right).
    \end{aligned}
\end{equation*}
\end{prop}

\begin{proof}
We focus on the case $\eta = 0$ and remark that the situation of $\eta > 0$ readily follows by multiplying the functions by $\mre^{\eta t}$ and making use of the maximal regularity of $A_0^m + \eta$ as well as the exponential stability in view of the invertibility.
We also observe that the $d$-component in \eqref{eq:linearization global} is decoupled from the remaining part, and the desired estimates in the components $v$, $\ell$ and $\omega$ hence follow from \autoref{prop:max reg on average zero space}, so it only remains to investigate the $d$-component.
We split the problem in several parts and first consider the lifting of the initial values.
To this end, we introduce $\varphi^1$ solving
\begin{equation}\label{eq:lifting initial values}
    \begin{aligned}
        \partial_t \varphi^1 + \mu' \varphi^1 - \Delta \varphi^1 &= 0, &&\tin (0,\infty) \times \cF_0,\\
        \partial_\nu \varphi^1 &= 0, &&\ton (0,\infty) \times \partial \cF_0,\\
        \varphi^1(0) &= d_0 - d_*, &&\tin \cF_0.
    \end{aligned}
\end{equation}
In \eqref{eq:lifting initial values}, $\mu' > 0$ is sufficiently large to guarantee the existence and uniqueness of a strong solution 
\begin{equation*}
    \varphi^1 \in \rL^p(0,\infty;\rW_{\rN}^{2,q}(\cF_0)) \cap \rW^{1,p}(0,\infty;\rL^q(\cF_0))
\end{equation*}
by maximal regularity of the Neumann Laplacian, see e.g.\ \cite[Theorem~8.2]{DHP:03}.
Performing the splitting $\varphi^1 = \varphi_m^1 + \varphi_\avg^1$ as introduced in \eqref{eq:splitting in mean zero and avg}, we infer that $\varphi_\avg^1$ satisfies
\begin{equation}\label{eq:varphi1 avg}
    \varphi_\avg^1(t) = (d_{0,\avg} - d_*) \mre^{-\mu' t}.
\end{equation}
This also implies $\varphi_\avg^1 \in \rL^r(0,\infty)$ for all $r \in [1,\infty]$.
Moreover, for $g_2 = g_{2,m} + g_{2,\avg}$, we set 
\begin{equation}\label{eq:varphi2}
    \varphi^2(t) \coloneqq \int_0^t g_{2,\avg}(s) + \mu' \varphi_\avg^1(s) \rd s.
\end{equation}
The assumption $g_{2,\avg} \in \rL^1(0,\infty)$ and the above observation on $\varphi_\avg^1$ yield that $\varphi^2$ is well-defined, and
\begin{equation}\label{eq:Linfty est varphi2}
    |\varphi^2(t)| \le C \left(\| g_{2,\avg} \|_{\rL^1(0,\infty)} + \| \varphi_\avg^1 \|_{\rL^1(0,\infty)}\right)
\end{equation}
for all $t \in (0,\infty)$.
Let us also observe that $\varphi^2 = \varphi_\avg^2$ because $\varphi^2$ is constant in space by definition.
For $\varphi^1$ as in \eqref{eq:lifting initial values}, $\varphi^2$ as in \eqref{eq:varphi2} and $d$ solving the $d$-part associated to \eqref{eq:linearization global}, i.e., $\partial_t d - \Delta d = g_2$ in $(0,\infty) \times \cF_0$, $\partial_\nu d = 0$ on $(0,\infty) \times \partial \cF_0$ and $d(0) = d_0 - d_*$ in $\cF_0$, we define $\Tilde{\varphi} \coloneqq d - \varphi^1 - \varphi^2$.
By construction, it follows that $\Tilde{\varphi}_\avg = 0$.
In addition, we obtain
\begin{equation*}
    \frac{1}{|\cF_0|} \int_{\cF_0} (d_0(y) - d_*) \rd y = d_{0,\avg} - d_*.
\end{equation*}
In view of $\rB_{qp}^{2-\nicefrac{2}{p}}(\cF_0) \hookrightarrow \rL^\infty(\cF_0)$ thanks to \eqref{eq:cond p and q}, we infer that
\begin{equation}\label{eq:est of d_0avg - d*}
    |d_{0,\avg} - d_* | \le \| d_0 - d_* \|_{\rL^\infty(\cF_0)} \le C \| d_0 - d_* \|_{\rB_{qp}^{2-\nicefrac{2}{p}}(\cF_0)}.
\end{equation}
In the end, we use the splitting $d = d_m + d_\avg$ from \eqref{eq:splitting in mean zero and avg} for the solution $d$ of the $d$-component of~\eqref{eq:linearization global}.
The above observations yield that $d_\avg = \varphi_\avg^1 + \varphi^2$, so \eqref{eq:varphi1 avg}, \eqref{eq:Linfty est varphi2} and \eqref{eq:est of d_0avg - d*} lead to
\begin{equation}\label{eq:est Linfty davg}
    \| d_\avg \|_{\rL^\infty(0,\infty)} \le C\left(\| d_0 - d_* \|_{\rB_{qp}^{2-\nicefrac{2}{p}}(\cF_0)} + \| g_{2,\avg} \|_{\rL^1(0,\infty)}\right).
\end{equation}
We calculate $\partial_t d_\avg = g_{2,\avg}$.
By definition of $g_{2,\avg}$, we find $\| g_{2,\avg} \|_{\rL^p(0,\infty)} \le C \| g_2 \|_{\rL^p(0,\infty;\rL^q(\cF_0))}$, so
\begin{equation}\label{eq:est Linfty partial_t davg}
    \| \partial_t d_\avg \|_{\rL^p(0,\infty)} \le C \| g_2 \|_{\F_\infty^d}.
\end{equation}
A short computation yields that $d_m = d - d_\avg$ solves $\partial_t d_m - \Delta d_m = g_{2,m}$ in $(0,\infty) \times \cF_0$, $\partial_\nu d_m = 0$ on $(0,\infty) \times \partial \cF_0$ and $d_m(0) = d_{0,m}$ in $\cF_0$.
From the maximal $\rL^p(0,\infty)$-regularity of $-\Delta_{\rN}$ on $\rL_0^q(\cF_0)$, \eqref{eq:est of d_0avg - d*} and the aforementioned estimate
\begin{equation*}
    \| g_{2,\avg} \|_{\rL^p(0,\infty)} \le C \| g_2 \|_{\rL^p(0,\infty;\rL^q(\cF_0))},
\end{equation*}
we derive that
\begin{equation}\label{eq_est d_m}
    \| d_m \|_{\E_\infty^d} \le C\left(\| g_{2,m} \|_{\F_\infty^d} + \| d_0 - d_{0,\avg} \|_{\rB_{qp}^{2-\nicefrac{2}{p}}(\cF_0)}\right) \le C\left(\| g_2 \|_{\F_\infty^d} + \| d_0 - d_* \|_{\rB_{qp}^{2-\nicefrac{2}{p}}(\cF_0)}\right).
\end{equation}
The assertion of the proposition follows by concatenating \eqref{eq:est Linfty davg}, \eqref{eq:est Linfty partial_t davg} and \eqref{eq_est d_m} thanks to the above reduction arguments. 
\end{proof}

\subsection{Estimates of the nonlinear terms}\label{ssec:ests of the nonlinear terms}
\ 

First, we introduce the nonlinear terms, resulting from the linearization in \eqref{eq:linearization global} with $\mu = 0$ 
\begin{equation}\label{eq:nonlin terms global}
    \begin{aligned}
        G_1(z,p) &= -\cM v + (\cL_1 - \Delta)v - \cN(v) - (\cG  - \nabla)p - \cB(d)d,\\
        G_2(z) &= (\cL_2 - \Delta) d - \nabla d \cdot \partial_t Y + (v \cdot \nabla)d + |\nabla d|^2 d + |\nabla d|^2 d_*,\\ 
        G_3(z,p) &= - \mS \omega \times \ell - \int_{\partial \cS_0} Q^\top \nabla d [\nabla d]^\top Q N \rd \Gamma - \int_{\partial \cS_0} Q^\top (Q \D v Q^\top - 2 \D v)Q N \rd \Gamma,\\
        G_4(z,p) &= \omega \times (J_0 \omega) - \int_{\partial \cS_0} y \times Q^\top \nabla d [\nabla d]^\top Q N \rd \Gamma - \int_{\partial \cS_0} y \times Q^\top (Q \D v Q^\top - 2 \D v)Q N \rd \Gamma.
    \end{aligned}
\end{equation}
The aim of this section is to provide estimates of $G_1$, $G_2$, $G_3$ and $G_4$ tailored to the situation of \autoref{prop:max reg type result global}.
In the sequel, we fix $\eta \in (0,\eta_0)$, where $\eta_0 > 0$ is the constant from \eqref{eq:exp stability restricted semigroup}.
For $\tE_\infty^m$ as introduced in \eqref{eq:data and sol space average zero}, the space $\tcK$ is defined such that $(z,p) = (v,d,\ell,\omega,p) \in \tcK$ if and only if
\begin{equation*}
    \mre^{\eta(\cdot)}(v,d_m,\ell,\omega,p) \in \tE_\infty^m, \enspace d_\avg \in \rL^\infty(0,\infty)^3, \tand \mre^{\eta(\cdot)} \partial_t d_\avg \in \rL^p(0,\infty)^3,
\end{equation*}
and it is equipped with the norm
\begin{equation*}
    \| (v,d,\ell,\omega,p) \|_{\tcK} \coloneqq \left\| (\mre^{\eta(\cdot)}v,\mre^{\eta(\cdot)}d_m,\mre^{\eta(\cdot)}\ell,\mre^{\eta(\cdot)}\omega,\mre^{\eta(\cdot)}p) \right\|_{\tE_\infty^m} + \left\| d_\avg \right\|_{\rL^\infty(0,\infty)} + \left\| \mre^{\eta(\cdot)} \partial_t d_\avg \right\|_{\rL^p(0,\infty)}.
\end{equation*}
Moreover, for $\eps > 0$, we denote the ball in $\tcK$ with center zero and radius $\eps$ by $\tcK_\eps$.

Similarly as in \autoref{sec:proof local well-posedness}, it is important to estimate the terms related to the diffeomorphisms $X$ and $Y$.
In fact, it is necessary to refine the estimates established in \autoref{lem:props of the transform} and \autoref{lem:ests covariant contrvariant christoffel}.
We address the adjusted estimates of the terms related to the transform in the following lemma.
The abbreviations of the norms correspond to the situation of the time interval $(0,\infty)$ in the sequel.

\begin{lem}\label{lem:ests trafo global}
Let $X$, $Y$, $g^{ij}$ and $\Gamma_{jk}^i$ be as defined in \autoref{sec:change of var}, and consider $(v,d,\ell,\omega,p) \in \tcK_\eps$.
Then the following assertions are valid.
\begin{enumerate}[(a)]
    \item There exists a constant $C>0$, independent of $\eps$, such that $\| \rJ_X - \Id \|_{\infty,\infty} \le C \eps$.
    In particular, for $\eps < \nicefrac{1}{2C}$, it holds that $\rJ_X$ is invertible on $(0,\infty) \times \cF_0$.
    \item There is $C > 0$ such that $\| \partial_i X_j \|_{\infty,\infty}$, $\| \partial_i Y_j \|_{\infty,\infty}$, $\| g^{ij} \|_{\infty,\infty}$, $\| g_{ij} \|_{\infty,\infty}$, $\| \Gamma_{jk}^i \|_{\infty,\infty} \le C$ holds true for all $i,j,k \in \{1,2,3\}$.
    \item There exists $C > 0$ such that for all $i,j,k \in \{1,2,3\}$, we have
    \begin{equation*}
        \begin{aligned}
            &\quad\| \partial_t Y \|_{\infty,\infty} + \| \partial_j \partial_t X_i \|_{\infty,\infty} + \| \rJ_Y - \Id \|_{\infty,\infty} + \| \partial_i g^{jk} \|_{\infty,\infty} + \| \partial_i g_{jk} \|_{\infty,\infty}\\
            &\quad + \| \Gamma_{jk}^i \|_{\infty,\infty} + \| g^{ij} - \delta_{ij} \|_{\infty,\infty} + \| \partial_i \partial_j X_k \|_{\infty,\infty} + \| \partial_i \partial_j Y_k \|_{\infty,\infty}\le C \eps.
        \end{aligned}
    \end{equation*}
\end{enumerate}
\end{lem}

\begin{proof}
From \eqref{eq:IVP X}, we deduce that
\begin{equation}\label{eq:shape of X global}
    X(t,y) = y + \int_0^t b(s,X(s,y)) \rd s.
\end{equation}
Moreover, let us observe that $\| Q_i \|_{\rL^\infty(0,\infty)} \le C \| \omega_i \|_{\rL^\infty(0,\infty)} \le C$ in view of $(z,p) \in \tcK_\eps$.
Hence, the shape of $b$ as defined in \eqref{eq:rhs b} as well as $(z,p) \in \tcK_\eps$ imply that
\begin{equation}\label{eq:est of b}
    \| \partial^\beta b_i \|_{\infty,\infty} \le C \| Q_i \|_\infty (\| \ell_i \|_\infty + \| \omega_i \|_\infty) \le C \mre^{-\eta t}(\mre^{\eta t}\left(\| \ell \|_{\rW^{1,p}(0,\infty)} + \| \omega \|_{\rW^{1,p}(0,\infty)}\right) \le C \mre^{-\eta t} \eps
\end{equation}
for every multi-index $\beta$ such that $0 \le |\beta| \le 3$.
Inserting the estimate \eqref{eq:est of b} into \eqref{eq:shape of X global}, we obtain the estimate $\| \rJ_X - \Id \|_{\infty,\infty} \le C \eps \int_0^\infty \mre^{-\eta s} \rd s \le C \eps$, so the first part of the assertion of~(a) is shown.
The second part is a direct consequence of a Neumann series argument.
The assertion of~(b) follows similarly as~(a), using the shape of $X$ from \eqref{eq:shape of X global}, and an analogous strategy can be used for $Y$.
The boundedness of the contravariant and covariant tensors as well as of the Christoffel symbol also readily follows therefrom.
Let us also refer to \autoref{lem:props of the transform}(a) and \autoref{lem:ests covariant contrvariant christoffel} or \cite[Section~6.1]{GGH:13}.

It remains to prove~(c).
To this end, we first observe that $\partial_t Y = b^{(Y)}$, where $b^{(Y)}$ is as in \eqref{eq:rhs bY}.
The desired smallness of $\partial_t Y$ then follows by the boundedness of $\rJ_X$ and its inverse $\rJ_X^{-1}$ by (a) and (b) in conjunction with the estimate of $b$ from \eqref{eq:est of b}.
Similarly, we deduce the smallness of $\partial_j \partial_t X_k$, this time using that $\partial_t X = b$ together with the estimate of $\partial_j b$ resulting from \eqref{eq:est of b}.
With regard to the estimate of $\rJ_Y - \Id$, as in \eqref{eq:shape of X global}, we obtain $Y(t,x) = x + \int_0^t b^{(Y)}(s,Y(s,x)) \rd s$.
As in the proof of~(a), we derive therefrom the desired estimate of $\| \rJ_Y - \Id \|_{\infty,\infty}$.
Next, let us observe that the $\| \cdot \|_{\infty,\infty}$-estimates of $\partial_i g^{jk}$ and $\partial_i g_{jk}$ are a consequence of the estimates of $\partial_i \partial_j Y_k$ and $\partial_i \partial_j X_k$, respectively, in conjunction with the boundedness of $\partial_i Y_j$ and $\partial_i X_j$ obtained in~(b).
The estimate of $\Gamma_{jk}^i$ is in turn implied by the boundedness of $g^{jk}$ shown in~(b) as well as the estimates of $\partial_i g^{jk}$.
Therefore, the main task is to estimate $\partial_i \partial_j X_k$ and $\partial_i \partial_j Y_k$ in $\| \cdot \|_{\infty,\infty}$.
In that respect, \eqref{eq:shape of X global} and \eqref{eq:est of b} yield
\begin{equation*}
    \partial_i \partial_j X_k = \int_0^t \partial_i \partial_j b_k(s,X(s,y)) \rd s, \tand \| \partial_i \partial_j X_k \|_{\infty,\infty} \le C \eps \int_0^\infty \mre^{-\eta s} \rd s \le C \eps.
\end{equation*}
By the boundedness of $\rJ_X^{-1}$, the analogue estimate of $\partial_i \partial_j Y_k$ follows.
The last term to be estimated is $g^{ij} - \delta_{ij}$ which we expand as
\begin{equation*}
    g^{ij} - \delta_{ij} = (\partial_k Y_i - \delta_{ki}) \partial_k Y_j + \delta_{ki} \partial_k Y_j - \delta_{ij} = (\partial_k Y_i - \delta_{ki}) \partial_k Y_j + \partial_i Y_j - \delta_{ij}.
\end{equation*}
The desired estimate follows by the above estimate of $\rJ_Y - \Id$ and the boundedness of $\partial_i Y_j$ from~(b).
\end{proof}

In the sequel, we consider $\eps_0 > 0$ sufficiently small so that $\rJ_X$ is invertible for all $\eps \in (0,\eps_0)$ by \autoref{lem:ests trafo global}(a).
The estimates of the terms related to the diffeomorphisms $X$ and $Y$ from the above \autoref{lem:ests trafo global} are the main ingredient for the subsequent estimates of the nonlinear terms $G_1$, $G_2$, $G_3$ and $G_4$ from \eqref{eq:nonlin terms global}.
In contrast to \autoref{ssec:estimates nonlinear terms}, we also need to provide estimates of $G_{2,\avg}$ with regard to the maximal regularity type result \autoref{prop:max reg type result global}.

\begin{prop}\label{prop:self map type est global}
Assume that $p,q \in (1,\infty)$ satisfy \eqref{eq:cond p and q}.
Then there is a constant $C > 0$ such that for every $\eps \in (0,\eps_0)$ and for every $(v,d,\ell,\omega,p) \in \tcK_\eps$, it holds that
\begin{equation*}
    \left\| \mre^{\eta(\cdot)}(G_1(z,p),G_2(z),G_3(z,p),G_4(z,p)) \right\|_{\F_\infty} + \left\| G_{2,\avg}(z) \right\|_{\rL^1(0,\infty)} \le C \eps^2.
\end{equation*}
\end{prop}

\begin{proof}
With regard to \eqref{eq:nonlin terms global}, it is possible to estimate all terms separately.
We only discuss a few terms in detail.
Making use of \autoref{lem:ests trafo global} and $(z,p) \in \tcK_\eps$, we first estimate
\begin{equation*}
    \begin{aligned}
        &\quad \| \mre^{\eta(\cdot)} \cM  v \|_{p,q}\\
        &\le C \sup_{i,j,k}\left(\| \partial_t Y_j \|_{\infty,\infty} \| \mre^{\eta(\cdot)} \partial_j v_i \|_{p,q} + (\| \Gamma_{jk}^i \|_{\infty,\infty} \| \partial_t Y_k \|_{\infty,\infty} + \| \partial_k Y_i \|_{\infty,\infty} \| \partial_j \partial_t X_k \|_{\infty,\infty}) \| \mre^{\eta(\cdot)} v_j \|_{p,q}\right)\\
        &\le C\left(\eps \| \mre^{\eta(\cdot)} z \|_{\E_\infty} + (\eps^2 + \eps) \| \mre^{\eta(\cdot)} z \|_{\E_\infty}\right) \le C \eps^2.
    \end{aligned}
\end{equation*}
A similar argument shows that
\begin{equation*}
    \| \mre^{\eta(\cdot)} (\cL_1 - \Delta) v \|_{p,q} \le C(\eps \| \mre^{\eta(\cdot)} z \|_{\E_\infty} + \eps \| \mre^{\eta(\cdot)} z \|_{\E_\infty} + \eps^2 \| \mre^{\eta(\cdot)} z \|_{\E_\infty}) \le C \eps^2.
\end{equation*}
As \autoref{lem:embedding of max reg space} is also valid for $\R_+$ as underlying time interval, and $\| z \|_{\E_\infty} \le C \mre^{-\eta t} \eps \le C \eps$ by $(z,p) \in \tcK_\eps$, we get by similar arguments as above
\begin{equation*}
    \| \mre^{\eta(\cdot)} \cN(v) \|_{p,q} \le \| \mre^{\eta(\cdot)} (v \cdot \nabla)v \|_{p,q} + C \sup_{i,j,k} \| \mre^{\eta(\cdot)} \Gamma_{jk}^i v_j v_k \|_{p,q} \le C \eps^2.
\end{equation*}
It also follows that $\| \mre^{\eta(\cdot)} (\cG  - \nabla)p \|_{p,q} \le C \eps^2$ by \autoref{lem:ests trafo global} and the representation of $\nabla p$.
In order to complete the estimate of $\| \mre^{\eta(\cdot)} G_1 \|_{\E_\infty^v}$, it remains to provide the estimate of $\cB(d)d$ as introduced in \autoref{sec:change of var}.
We start with the term $\cL_2$ and conclude from \autoref{lem:ests trafo global} and $\partial^\beta d = \partial^\beta d_m$ for every multi-index $\beta$ with $1 \le |\beta| \le 2$ the estimate
\begin{equation*}
    \| \mre^{\eta(\cdot)} \cL_2 d \|_{p,q} \le C(\| g^{jk} \|_{\infty,\infty} \| \mre^{\eta(\cdot)} d_m \|_{\E_\infty^d} + \| \Delta Y_k \|_{\infty,\infty} \| \mre^{\eta(\cdot)} d_m \|_{\E_\infty^d}) \le C \eps.
\end{equation*}
Together with another application of \autoref{lem:ests trafo global}, the above observation that $d_\avg$ is constant in space and a similar estimate of $\| d_m \|_{\E_\infty^d}$ as above in the context of the fluid velocity, this results in
\begin{equation*}
    \left\| \mre^{\eta(\cdot)} \sum\limits_{l=1}^3 \left((\cL_2 h)_{l}\sum\limits_{m=1}^3 {\partial_m d_{l}}{\partial_i Y_m}  \right) \right\|_{p,q} \le C \eps \| d_m \|_{\E_\infty^d} \le C \eps^2.
\end{equation*}

The desired estimate of $\| \mre^{\eta(\cdot)} \cB(d)d \|_{\E_\infty^d}$ follows by estimating the second part in an analogous way by
\begin{equation*}
    \begin{aligned}
        &\quad \left\| \mre^{\eta(\cdot)} \sum\limits_{k,l,m=1}^3 \left( {\partial_m d_{l}}{\partial_k Y_m}\right)\left(\sum\limits_{j=1}^3 (\partial_j \partial_m d_{l})(\partial_k Y_j \partial_i Y_m) +  (\partial_m d_l )(\partial_k \partial_i Y_m )\right) \right\|_{p,q}\\
        &\le C \| d_m \|_{\E_\infty^d} \left(\| \mre^{\eta(\cdot)} d_m \|_{\E_\infty^d} + \| \mre^{\eta(\cdot)} d_m \|_{\E_\infty^d} \| \partial_k \partial_i Y_m \|_{\infty,\infty}\right) \le C \eps^2.
    \end{aligned}
\end{equation*}
With regard to the estimate of $\| \mre^{\eta(\cdot)} G_2 \|_{\F_\infty^d}$, we observe that $(\cL_2 - \Delta)d$ and $(v \cdot \nabla) d$ can be estimated similarly as the above terms, while
\begin{equation*}
    \| \mre^{\eta(\cdot)} \nabla d \cdot \partial_t Y \|_{p,q} \le C \| \mre^{\eta(\cdot)} d_m \|_{\E_\infty^d} \| \partial_t Y \|_{\infty,\infty} \le C \eps^2
\end{equation*}
follows from \autoref{lem:ests trafo global}.
By \autoref{lem:embedding of max reg space} on $\R_+$ and $\| d \|_{\infty,\infty} \le \| d_m \|_{\E_\infty^d} + \| d_\avg \|_\infty \le 2 \eps$ resulting from $(z,p) \in \tcK_\eps$, we get
\begin{equation*}
    \left\| \mre^{\eta(\cdot)} |\nabla d|^2 d \right\|_{p,q} \le C \| d \|_{\infty,\infty} \| \nabla d_m \|_{\infty,\infty} \left\| \mre^{\eta(\cdot)} \nabla d_m \right\|_{p,q} \le C \eps \| d_m \|_{\E_\infty^d} \left\| \mre^{\eta(\cdot)} d_m \right\|_{\E_\infty^d} \le C \eps^3.
\end{equation*}
The corresponding estimate for $|\nabla d|^2 d_*$ by $C \eps^2$ follows in the same way.
In conclusion, the estimate of $\| \mre^{\eta(\cdot)} G_2 \|_{\F_\infty^d}$ is proved.
Concerning the estimates of $G_3$ and $G_4$, it readily follows that $\| \mre^{\eta(\cdot)} m \omega \times \ell \|_p \le C \eps^2$ and likewise for $\| \mre^{\eta(\cdot)} \omega \times (J_0 \omega) \|_p$, so it remains to estimate the surface integrals.
Let us observe that the surface integral in $G_4$ can be treated completely analogously as the one in $G_3$, so it suffices to deal with the latter one.
For the first addend in the surface integrals, we proceed as in \eqref{eq:est d part surface ints}, recall the embedding \eqref{eq:emb max reg space into Linfty H32} and invoke the boundedness of $Q_i$ as well as $Q_i^\top$ as established in the proof of \autoref{lem:ests trafo global} to obtain
\begin{equation*}
    \begin{aligned}
        \left\| \mre^{\eta(\cdot)} \int_{\partial \cS_0} Q^\top \nabla d [\nabla d]^\top Q N \rd \Gamma \right\|_p
        &= \left\| \mre^{\eta(\cdot)} \int_{\partial \cS_0} Q^\top \nabla d_m [\nabla d_m]^\top Q N \rd \Gamma \right\|_p\\
        &\le C \| d_m \|_{\rL^\infty(0,\infty;\rH^{\nicefrac{3}{2}}(\cF_0))} \| \mre^{\eta(\cdot)} d_m \|_{\rL^p(0,\infty;\rH^{\nicefrac{3}{2}}(\cF_0))}\\
        &\le C \| d_m \|_{\E_\infty^d} \| \mre^{\eta(\cdot)} d_m \|_{\E_\infty^d} \le C \eps^2.
    \end{aligned}
\end{equation*}
In order to estimate the other addend, we rewrite $Q \D v Q^\top - \D v = (Q - \Id) \D v Q^\top + \D v(Q^\top - \Id)$.
For an estimate by $C \eps^2$, it is thus necessary to show that $Q - \Id$ as well as $Q^\top - \Id$ can be bounded by $\eps$.
By the definition of the transform in \autoref{sec:change of var}, we obtain $Q^\top(t) = \Id + \int_0^t M(s) Q^\top(s) \rd s = \Id + \int_0^t \omega(s) \times \cdot \rd s$.
As a result, the estimate
\begin{equation*}
    \| Q^\top - \Id \|_\infty \le \int_0^\infty \mre^{-\eta s} \| \mre^{\eta(\cdot)} \omega \|_\infty \rd s \le C \| \mre^{\eta(\cdot)} \omega \|_{\E_\infty^\omega} \le C \eps
\end{equation*}
is valid.
A similar estimate can be shown for $\| Q - \Id \|_\infty$.
Combining the above arguments and employing \eqref{eq:joint est surface ints}, we conclude
\begin{equation*}
    \left\| \mre^{\eta(\cdot)} \int_{\del \cS_0} Q^\top (Q \D  v Q^\top - \D v) Q N \rd \Gamma \right\|_p \le C \eps \bigl\| \mre^{\eta(\cdot)} v \bigr\|_{\rL^p(0,\infty;\rW^{1 + \nicefrac{2}{q},q}(\cF_0))} \le C \eps \bigl\| \mre^{\eta(\cdot)} v \bigr\|_{\E_\infty^v} \le C \eps,
\end{equation*}
completing the estimates of $G_3$ and also of $G_4$ by analogy.
For the estimate of $\| G_{2,\avg} \|_{\rL^1(0,\infty)}$, note that
\begin{equation}\label{eq:G1avg}
    G_{2,\avg} = \frac{1}{|\cF_0|} \int_{\cF_0} (\cL_2 - \Delta)v - \nabla d \cdot \partial_t Y + (v \cdot \nabla)d + |\nabla d|^2 d + |\nabla d|^2 d_* \rd y.
\end{equation}
It follows by the boundedness of the spatial domain $\cF_0$ and the above estimate of $(\cL_2 - \Delta)d$ that
\begin{equation*}
    \begin{aligned}
        \left\| \frac{1}{|\cF_0|} \int_{\cF_0} (\cL_2 - \Delta)d \rd y \right\|_{\rL^1(0,\infty)}
        &\le C \int_0^\infty \| (\cL_2 - \Delta)d \|_{\rL^q(\cF_0)} \rd s\\
        &\le C \| \mre^{\eta(\cdot)} (\cL_2 - \Delta) d \|_{p,q} \int_0^\infty \mre^{-\eta s p'} \rd s \le C \eps^2,
    \end{aligned}
\end{equation*}
where we denote by $p'$ the H\"older conjugate of $p$.
In the same way, we obtain the corresponding estimates for the remaining terms in \eqref{eq:G1avg}.
In total, the assertion of the proposition follows.
\end{proof}

The next proposition on Lipschitz estimates follows in a similar way as in \autoref{ssec:estimates nonlinear terms} by also using Lipschitz estimates of the transform as well as the boundedness and decay established in \autoref{lem:ests trafo global}.

\begin{prop}\label{prop:lipschitz est global}
Assume that $p,q \in (1,\infty)$ satisfy \eqref{eq:cond p and q}.
Then there exists a constant $C>0$ such that for every $\eps \in (0,\eps_0)$ and for $(v_1,d_1,\ell_1,\omega_1,p_1)$, $(v_2,d_2,\ell_2,\omega_2,p_2) \in \tcK_\eps$, it follows that
\begin{equation*}
    \begin{aligned}
        &\quad \left\| \mre^{\eta(\cdot)} (G_1(z_1,p_1) - G_1(z_2,p_2),G_2(z_1) - G_2(z_2),G_3(z_1,p_1) - G_3(z_2,p_2),G_4(z_1,p_1) - G_4(z_2,p_2)) \right\|_{\F_\infty}\\
        &\quad + \left\| G_{2,\avg}(z_1) - G_{2,\avg}(z_2) \right\|_{\rL^1(0,\infty)} \le C \eps \left\| (z_1,p_1) - (z_2,p_2) \right\|_{\tcK}.
    \end{aligned}
\end{equation*}
\end{prop}

\subsection{Proof of \autoref{thm:global wp close to equilibria}}\label{ssec:proof of global wp close to equilibria}
\

As in the situation of the local well-posedness in \autoref{sec:proof local well-posedness}, we will first prove \autoref{thm:global wp close to equilibria reformulated} stated in the reference configuration and then deduce \autoref{thm:global wp close to equilibria} therefrom.

\begin{proof}[Proof of \autoref{thm:global wp close to equilibria reformulated}]
We start by defining the map for which we will show that it admits a unique fixed point.
We introduce $\Psi \colon \tcK_\eps \to \tcK$, where given $(z,p) = (v,d,\ell,\omega,p)$, we denote by $\Psi(z,p)$ the solution to \eqref{eq:linearization global} with right-hand sides $G_1(z,p)$, $G_2(z,p)$, $G_3(z,p)$ and $G_4(z,p)$.
It then follows by \autoref{prop:max reg type result global} that $\Psi$ is well-defined.
For $\eps \in (0,\eps_0)$, where $\eps_0 > 0$ is made precise after \autoref{lem:ests trafo global}.
Moreover, for $(z,p),(z_1,p_1),(z_2,p_2) \in \tcK_\eps$, we conclude from \autoref{prop:max reg type result global}, \autoref{prop:self map type est global}, \autoref{prop:lipschitz est global} and $\| (v_0,d_0-d_*,\ell_0,\omega_0) \|_{\rX_\gamma} \le \delta$ that
\begin{equation*}
    \left\| \Psi(z,p) \right\|_{\tcK} \le \Cexp (\delta + C \eps^2), \tand \left\| \Psi(z_1,p_1) - \Psi(z_2,p_2) \right\|_{\tcK} \le \Cexp C \eps.
\end{equation*}
Choosing $\delta \le \nicefrac{1}{2 \Cexp}$ as well as $\eps \le \min\left\{\eps_0,\nicefrac{1}{2 \Cexp C},\nicefrac{1}{2C},\sqrt{\nicefrac{\delta}{C}}\right\}$, we derive that $\Psi$ is a self-map and a contraction.
Hence, there exists a unique fixed point $(z',p') = (v',d',\ell',\omega',p')$ by the contraction mapping principle.
This fixed point is then the desired unique global strong solution to the transformed interaction problem \eqref{eq:cv1}--\eqref{eq:cv3}.
The regularity of the functions and the decay properties follow by construction of $\tcK$.
Moreover, the regularity of the transform $X$ is a consequence of \autoref{lem:props of the transform}.
\end{proof}

\begin{proof}[Proof of \autoref{thm:global wp close to equilibria}]
\autoref{thm:global wp close to equilibria reformulated} especially yields $(\ell,\omega) \in \rW^{1,p}(0,\infty)^6$, and $h'$, $\Omega$ and the diffeomorphism $X$ can be recovered therefrom as revealed in \autoref{rem:procedure to determine diffeos}.
In order to derive the solution to the original problem, we apply the backward change of coordinates as introduced in \autoref{sec:change of var}, and it readily follows that this is a solution to the interaction problem \eqref{eq:LC-fluid equationsiso}--\eqref{eq:initial} in the desired regularity class.
It remains to show that their is no collision, i.e., $\dist(\cS(t),\partial \cO) > \nicefrac{r}{2}$ is valid for all $t \in [0,\infty)$.
From \eqref{eq:solid dom} and the above estimates, it follows for $\delta$ and $\eps$ sufficiently small that $\dist(\cS(t),\cS_0) \le \| h(t) \|_{\R^3} +  \| Q(t) - \Id \|_{\R^{3 \times 3}} < \nicefrac{r}{2}$ for all $t \in [0,\infty)$.
By virtue of $\dist(\cS_0,\partial \cO) > r$, no collision can occur.
\end{proof}

\medskip 

{\bf Acknowledgements}
Tim Binz would like to thank DFG for support through project 538212014.
Felix Brandt~and Matthias Hieber~acknowledge the support by DFG project FOR~5528, while Arnab Roy~would like to express his gratitude to the Alexander von Humboldt-Stiftung / Foundation for their support.

\bibliography{lcfsi}
\bibliographystyle{siam}

\end{document}